\documentclass{ws-idaqp}
\usepackage{graphicx}
\usepackage[super]{cite}

\newcommand{\E}{{\mathbb{E}}}
\newcommand{\R}{{\mathbb{R}}}

\begin{document}

\markboth{Qingsong Wang, A. A. Dorogovtsev, K. V. Hlyniana, Naoufel Salhi}{The Geometry of Gaussian Random Curves}


\title{The Geometry of Gaussian Random Curves}

\author{Qingsong Wang}

\address{School of Mathematics, Jilin University, Changchun, 130012, P.~R. China.\\
qswang21@mails.jlu.edu.cn}

\author{A. A. Dorogovtsev}

\address{Institute of Mathematics, National Academy of Sciences of Ukraine, Ukraine.\\
andrey.dorogovtsev@gmail.com}

\author{K. V. Hlyniana}

\address{School of Mathematics, Jilin University, Changchun, 130012, P.~R. China.\\
glinkate@gmail.com}

\author{Naoufel Salhi}

\address{Research laboratory of Stochastic Analysis, university of Tunis El Manar, Tunisia.\\
salhi.naoufel@gmail.com}
\maketitle

\begin{abstract}
In this paper, we investigate some geometric properties of non-smooth random curves within a stochastic flow. We consider a polygonal line $\Gamma(\vec{u}_{1},\cdots,\vec{u}_{n})$, which connects the points \(\vec{u}_{1},\cdots,\vec{u}_{n}\in{\mathbb{R}^{d}}\) and is inscribed in a Brownian trajectory. Subsequently, we estimate the probability that a polygonal line is almost inscribed in a Brownian trajectory. Next, we turn to the study of the self-intersection local time of Brownian motion and demonstrate the asymptotic result of its conditional expectation as the size of the polygonal line increases. Finally, taking such a Brownian trajectory as the initial curve, we let it evolve according to the solution of the equation with interaction. Then, we prove that its visitation density exhibits an intermittency phenomenon.
\end{abstract}



\section{Introduction}
Since the 1970s, stochastic flows have been intensively studied as a model for the turbulent motion of gas or liquid. For instance, several works \cite{1,2,3,4,5} have explored this area. To understand the mutual behavior of particles in a flow, one can consider the evolution of sets of passive tracers in the flow. In the paper \cite{6} , a smooth manifold was placed in the flow, and the change in its curvatures was discussed. In the paper \cite{7} , the authors considered the evolution of the mass distribution in a flow and obtained the asymptotics of its moments. In these works, the particles in the flow were passive. That is, their positions or mass distribution did not change the flow. These studies have laid a foundation for further exploration of more complex scenarios in stochastic flows. In the works of A. A. Dorogovtsev \cite{8,9,10,11} , a special class of stochastic flows with interactions was introduced. Such flows change their characteristics depending on the mass distribution of the particles they transport. This framework offers new possibilities for modeling the motion of different objects in a turbulent flow, particularly when the characteristics of an object influence the behavior of the flow. One of the most interesting applications is the modeling of the motion of long linear polymers. Usually, these molecules are modeled by the non-smooth random curves. The geometry of such curves can be described using their self-intersection local times, as reported in \cite{12} . Here, we propose to describe the shape of the random curve using a polygonal line inscribed in it. In this paper, we consider a polygonal line $\Gamma(\vec{u_{1}},\cdots,\vec{u_{n}})$ inscribed in a Brownian trajectory. Firstly, we estimate the probability that a polygonal line is almost inscribed in a Brownian trajectory. Secondly, we examine the behavior of a common characteristic of the Brownian trajectory when the polygonal line inscribed in it. Here, we propose to consider the self-intersection local time and investigate how this self-intersection local time evolves as the size of the polygonal line increases. Thirdly, taking such a Brownian trajectory as the initial curve, we let it evolve according to the solution of the equation with interaction. Then, we prove that its visitation densities exhibit an intermittency phenomenon as time tends to infinity.

\section{The geometry of Brownian trajectory}

In physical and chemical fields, random curves are used to describe polymers, whose configurations evolve over time. The reason\cite{Hol09} is that: first, there are flexible chemical bonds between monomers; second, there is the influence of external force (such as stochastic interactions with the molecules of the medium); and last, there is the interaction between the different parts of the polymer. One of the possible mathematical descriptions of a linear polymer is a random curve parameterized by $\vec{\xi}:[0,1]\rightarrow{\R^{d}}$, which is a continuous random process. Denote the random curve by $\gamma$. Usually $\vec{\xi}$ is non-smooth and $\gamma$ has complicated fractal structure. In this case, $\gamma$ can be characterized by its visitation measures.
\begin{definition}
The visitation measures of $\vec{\xi}$ are
\begin{small}
\begin{equation*}
\begin{aligned}
&\mu^{1}(\Delta)=\int_{0}^{1}\mathbf{1}_{\Delta}(\vec{\xi}(t))dt, \ \Delta\in\mathcal{B}(\R^{d}),\\
&\cdots\\
&\mu^{k}(\Delta)=\int_{0}^{1}\cdots\int_{0}^{1}\mathbf{1}_{\Delta}[(\vec{\xi}(t_{2})-\vec{\xi}(t_{1})),\cdots,(\vec{\xi}(t_{k})-\vec{\xi}(t_{k-1}))]dt_{1}\cdots{dt_{k}}, \  \Delta\in\mathcal{B}(\R^{d(k-1)}).
\end{aligned}
\end{equation*}
\end{small}
\end{definition}
Let $\Delta$ be a small ball with radius $\varepsilon$ and center at the origin $0$. Then, $\mu^{k}$ can be regarded as the measure of time which $\vec{\xi}$ spends in the neighborhood of (k-1)-fold intersections.

In the next example, we will investigate how the visitation measures evolve when the Brownian trajectory is transformed by the smooth function.

\begin{example}\label{thm3}
Consider the Brownian motion $w=(w_{1},\cdots,w_{d})$ in $\R^{d}$. Assume that its visitation measure $\mu^{k}$ has a density $l^{k}$ (see \cite{13,220} ). For example, when $k=1$ and $d=1$, as well as  $k=2$ and $d\leq 3$ the visitation measure has a density. Denote $d\mu^{k}=l^{k}d\lambda$, where $\lambda$ is the corresponding Lebesgue measure. Define a new process $\eta=\varphi(w)$, where $\varphi=(\varphi_{1},\cdots,\varphi_{d})$ is a continuously differentiable vector-valued function belonging to the function space $\mathcal{C}^{1}(\R^{d})$. The function $\varphi$ satisfies the following condition: for any $y_{1},\cdots,y_{d}\in \R^{d}$, form a matrix by taking $\nabla\varphi(y_{i})$ as the $i$-th row, and the determinant of this matrix is non-zero. (For example, for $z=(z_{1},\dots,z_{d})\in \R^{d}$, define $\varphi_{1}=z_{1}+2z_{2}+\cdots+dz^{2}_{d}$, $\varphi_{2}=2z_{1}+z_{2}+\cdots+dz^{2}_{d}$, $\cdots$, $\varphi_{d-2}=(d-2)z_{1}+2z_{2}+\cdots+dz^{2}_{d}$,  $\varphi_{d-1}=(d-1)z_{1}+2z_{2}+\cdots+dz_{d}$,  $\varphi_{d}=dz_{1}+2z_{2}+\cdots+z_{d}$, wherein the coefficients of function $\varphi_{1}$ range from $1$ to $d$, the coefficients of function $\varphi_{i}$ are obtained by swapping the first coefficient and the $i$-th coefficient of the coefficients of function $\varphi_{1}$.  Then $\varphi=(\varphi_{1},\cdots,\varphi_{d})$ will satisfies the condition.) For the process $\eta$, its visitation measure $\mu_{\eta}^{k}$ has a density $l_{\eta}^{k}$, our primary objective is to determine the density $l_{\eta}^{k}$.

For $x=(x_{1},\cdots,x_{d})\in \R^{d}$, the $l^{k}$, which coincides with local time (see \cite{27,28}), can be expressed as follows:
\begin{equation*}
\begin{aligned}
l^{k}(x)=\lim_{\varepsilon\rightarrow0}\int_{0}^{1}\cdots\int_{0}^{1}\prod_{j=1}^{k-1}f_{\varepsilon}(w(t_{j+1})-w(t_{j})-x)dt_{1}\cdots{d}t_{k},
\end{aligned}
\end{equation*}
where we can take 
\begin{equation*}
\begin{aligned}
f_{\varepsilon}(x)=\frac{1}{{\rm vol}(B(0,\varepsilon))}\mathbf{1}_{B(0,\varepsilon)}(x),
\end{aligned}
\end{equation*}
and ${\rm vol}(B(0,\varepsilon))$ denotes the volume of $B(0,\varepsilon)$, $B(0,\varepsilon)$ is the ball centered at $0$ with radius $\varepsilon$. If the visitation density of the process $\eta$ can also be found as a corresponding local time, we consider the following approximation,
\begin{equation*}
\begin{aligned}
l_{\eta}^{k}(x)&=\lim_{\varepsilon\rightarrow0}\int_{0}^{1}\cdots\int_{0}^{1}\prod_{j=1}^{k-1}f_{\varepsilon}(\varphi(w(t_{j+1}))-\varphi(w(t_{j}))-x)dt_{1}\cdots{d}t_{k}.
\end{aligned}
\end{equation*}
To show the relationship between $l^{k}$ and $l_{\eta}^{k}$, first, we investigate the indicator part of $f_{\varepsilon}(\varphi(w(t_{j+1}))-\varphi(w(t_{j}))-x)$,
\begin{equation}\label{2.24}
\begin{aligned}
&\mathbf{1}_{B(0,\varepsilon)}(\varphi(w(t_{j+1}))-\varphi(w(t_{j}))-x)\\
    =&\left\{
    \begin{array}{lcl}
        1, \quad \sqrt{\sum_{n=1}^{d}(\varphi_{n}(w(t_{j+1}))-\varphi_{n}(w(t_{j}))-x_{n})^{2}}\leq\varepsilon, \\
        0, \quad {\rm otherwise}.
      \end{array}
    \right.
\end{aligned}
\end{equation}
Now for every $\varphi_{n}:\R^{d}\rightarrow\R$, by the mean value theorem, there exists a $\theta^{i}_{n,j}(\omega)$ such that $\theta^{i}_{n,j}(\omega)=(1-c^{i}_{n,j}(\omega))w_{i}(t_{j+1})+c^{i}_{n,j}(\omega)w_{i}(t_{j})$ for some $c^{i}_{n,j}(\omega)\in[0,1]$. Denote $\theta_{n,j}=(\theta^{1}_{n,j}(\omega),\cdots,\theta^{d}_{n}(\omega))$, then
\begin{equation*}
\begin{aligned}
\sum_{n=1}^{d}(\varphi_{n}(w(t_{j+1}))-\varphi_{n}(w(t_{j}))-x_{n})^{2}=\sum_{n=1}^{d}(\langle \nabla\varphi_{n}(\theta_{n,j}),w(t_{j+1})-w(t_{j}) \rangle- x_{n})^{2}.
\end{aligned}
\end{equation*}
Let $D_{\varphi}(\theta)$ be the matrix such that
\begin{equation*}
    D_{\varphi}(\theta)=
    \left(
      \begin{array}{lcl}
        \nabla&\varphi_{1}(\theta_{1,j}) \\
        &\vdots  \\
        \nabla&\varphi_{d}(\theta_{d,j}) \\
      \end{array}
    \right).
    \end{equation*}
Here $D_{\varphi}(\theta)$ is a $d\times d$ square matrix, and $\det(D_{\varphi}(\theta))\neq0$, so the matrix $D_{\varphi}(\theta)$ is invertible. Thus we can rewrite (\ref{2.24}) as follows:
\begin{equation*}
\begin{aligned}
	&\mathbf{1}_{B(0,\varepsilon)}(\varphi(w(t_{j+1}))-\varphi(w(t_{j}))-x)\\
    =&\left\{
    \begin{array}{lcl}
        1, \quad \|D_{\varphi}(\theta)(w(t_{j+1})-w(t_{j}))^{\top}-x^{\top}\|\leq\varepsilon, \\
        0, \quad {\rm otherwise},
      \end{array}
    \right.\\
    =&\left\{
    \begin{array}{lcl}
        1, \quad (w(t_{j+1})-w(t_{j}))^{\top}-{D}^{-1}_{\varphi}(\theta)x^{\top}\in{D}^{-1}_{\varphi}(\theta){B}(0,\varepsilon), \\
        0, \quad {\rm otherwise},
      \end{array}
    \right.\\
    =&\mathbf{1}_{{D}^{-1}_{\varphi}(\theta){B}(0,\varepsilon)}((w(t_{j+1})-w(t_{j}))^{\top}-{D}^{-1}_{\varphi}(\theta)x^{\top}).
\end{aligned}
\end{equation*}
Denote $y=(w(t_{j+1})-w(t_{j}))^{\top}-{D}^{-1}_{\varphi}(\theta)x^{\top}$, 
\begin{equation*}
\begin{aligned}
&g_{\varepsilon}(y)
=\frac{1}{{\rm vol}({D}^{-1}_{\varphi}(\theta){B}(0,\varepsilon))}\mathbf{1}_{{D}^{-1}_{\varphi}(\theta){B}(0,\varepsilon)}(y),
\end{aligned}
\end{equation*}
then the density $l_{\eta}^{k}$ at ${D}^{-1}_{\varphi}(\theta)x^{\top}$ can be expressed as follows:
\begin{equation*}
\begin{aligned}
l_{\eta}^{k}({D}^{-1}_{\varphi}(\theta)x^{\top})=\lim_{\varepsilon\rightarrow0}\int_{0}^{1}\cdots\int_{0}^{1}\prod_{j=1}^{k-1}g_{\varepsilon}((w(t_{j+1})-w(t_{j}))^{\top}-{D}^{-1}_{\varphi}(\theta)x^{\top})dt_{1}\cdots{d}t_{k}.
\end{aligned}
\end{equation*}
So we have
\begin{equation*}
\begin{aligned}
l_{\eta}^{k}(x)=&\lim_{\varepsilon\rightarrow0}\int_{0}^{1}\cdots\int_{0}^{1}\prod_{j=1}^{k-1}g_{\varepsilon}((w(t_{j+1})-w(t_{j}))^{\top}-{D}^{-1}_{\varphi}(\theta)x^{\top})\\
               &(\frac{{\rm vol}({D}^{-1}_{\varphi}(\theta){B}(0,\varepsilon))}{{\rm vol}(B(0,\varepsilon))})^{k-1}dt_{1}\cdots{d}t_{k}\\
               =&\frac{l^{k}({D}^{-1}_{\varphi}(x)x^{\top})}{(\det{D}_{\varphi}(x))^{k-1}},
\end{aligned}
\end{equation*}
where ${D}_{\varphi}(x)$ is the Jacobi matrix of $\varphi$ at $x$.
\end{example}

Thus, we have shown how the visitation density evolves when the Brownian trajectory is changed by a smooth function $\varphi$.

One can model long linear polymers using the trajectories of a stochastic process \cite{Hol09,Bol02} .  In general, the shape of trajectories of stochastic process is very complicated, even for Brownian motion, because it is not differentiable. It is reasonable to consider the regularity in chaotic structure of the random curve. To understand it, we must discuss several questions. How can a simple geometric structure emerge in the non-smooth curve during its motion? When we consider the random curve in a stochastic flow, it will change its geometrical shape, so we need to understand how the visitation measures of the curve reflect this change, and what happens with the densities of visitation measures?

In the result presented in reference \cite{13} , for the two-dimensional case, it is shown that  $\mathbb{E}l^{2}({x})\rightarrow\infty$ as  ${x}\rightarrow 0$, where $l({x})$ represents the local time of the Brownian motion. This behavior bears a striking resemblance to a phenomenon known as intermittency. Specific structures emerge in random media due to instabilities. In these structures, a growing quantity attains extremely high values. Despite the rarity of these concentrations, they dominate the integral characteristics of the growing quantity, such as the mean value and the mean square value. The appearance of such structures is defined as intermittency \cite{Ya88} . In reference \cite{14} , X. Chen investigated the intermittency for Brownian motion in random media. Importantly, these concentration properties of energy are referred to as intermittency in the context of random processes and as a localization property for the polymer measure. In our description, the asymptotic result of the density $l^{2}$ for Brownian motion shows that this density is almost concentrated on the point $0$, which also means that the measure $\mu^{2}$ becomes close to the delta function.

If the curve is complicated, conventional geometric methods are often inadequate. However, if we can inscribe the curve with a polygonal line,we can leverage standard topological and geometric features to study the geometry of the polygonal line, thereby gaining insights into the characteristics of the original curve.

Suppose that the curve $\gamma$ is parametrized by the continuous function $\vec{\xi}:[0,1]\rightarrow{\R^{d}}$. For the points $\vec{u}_{1},\cdots,\vec{u}_{n}\in{\R^{d}}$, denote the polygonal line connecting $\vec{u}_{1}$ to $\vec{u}_{n}$ in order by $\Gamma(\vec{u}_{1},\cdots,\vec{u}_{n})$. 

\begin{definition}
The polygonal line $\Gamma(\vec{u}_{1},\cdots,\vec{u}_{n})$ is inscribed in the curve $\gamma$ if there exist $0\leq{t_{1}}<\cdots<t_{n}\leq1$ such that
\begin{equation*}
\begin{aligned}
&\vec{\xi}(t_{k})=\vec{u}_{k},\quad k=1,\cdots,n.
\end{aligned}
\end{equation*}
\end{definition}

\begin{definition}
The polygonal line $\Gamma(\vec{u}_{1},\cdots,\vec{u}_{n})$ is inscribed in the curve $\gamma$ up to the $\varepsilon>0$, if there exist $0\leq t_{1}<\cdots<t_{n}\leq1$, such that
\begin{equation*}
\begin{aligned}
&\vec{\xi}(t_{k})\in B(\vec{u}_{k},\varepsilon),\quad k=1,\cdots,n.
\end{aligned}
\end{equation*}
\end{definition}

\begin{definition}
Define the set $Z^{\varepsilon}=\cup_{k=1}^{N}B(\vec{z}_{k},\varepsilon)$, taking $\varepsilon>0$ such that
\begin{equation*}
\begin{aligned}
{B}(\vec{z}_{k},\varepsilon)\cap{B}(\vec{z}_{j},\varepsilon)=\phi,\quad k\neq{j},\quad k,j=1,\cdots,N.
\end{aligned}
\end{equation*}
Then the sequence of numbers $\alpha(1),\cdots,\alpha(n)$ can be associated with every $\Gamma(\vec{u}_{1},\cdots,\vec{u}_{n})$ which has the vertices from the set $Z^{\varepsilon}$ , here
\begin{equation*}
\begin{aligned}
\forall{j}=1,\cdots,n \quad \vec{u}_{j}\in{B}(\vec{z}_{\alpha(j)},\varepsilon).
\end{aligned}
\end{equation*}
In this case we say that the polygonal line $\Gamma$ has a type $\alpha(1),\cdots,\alpha(n)$ up to $\varepsilon$.
\end{definition}

\begin{definition}
The curve $\gamma$ has at least type $\alpha(1),\cdots,\alpha(n)$ up to $\varepsilon$ if the polygonal line of such type up to $\varepsilon$ can be inscribed in it up to $\varepsilon$.
\end{definition}

Next we will investigate the probability of a Brownian motion having at least of type $\alpha(1),\cdots,\alpha(n)$ for large $n$.

\begin{theorem}\label{thm2.8}
For any fixed $\{Z_{k}\}_{k\geq1}$, suppose that the random curve $\{\gamma(r),r\in[0,1]\}$ is the trajectory of a standard Brownian motion in $\R^{d}$, then 
\begin{equation*}
	\begin{aligned}
		\lim_{n\rightarrow+\infty}n^{-2} \ln {P}\{\gamma \ {\rm has} \ {\rm at} \ {\rm least}  \ {\rm type} \ \alpha(1),\cdots,\alpha(n) \ {\rm up} \ {\rm to} \ \varepsilon\}
		\geq-\frac{\|Z_{\alpha_{j\ast}}-Z_{\alpha_{j\ast-1}}\|^{2}}{2},
	\end{aligned}
\end{equation*}
where $Z_{\alpha_{j\ast}}-Z_{\alpha_{j\ast-1}}$ is selected from the set $\{Z_{\alpha_{j}}-Z_{\alpha_{j-1}}\}$ satisfing $\max_{j}\|Z_{\alpha_{j}}-Z_{\alpha_{j-1}}\|=\|Z_{\alpha_{j\ast}}-Z_{\alpha_{j\ast-1}}\|$.
\end{theorem}

\begin{proof}
Consider points $\frac{1}{n},\frac{2}{n},\cdots,\frac{n}{n}\in[0,1]$, since the random curve $\gamma(r)$ is the trajectory of a standard Brownian motion $\{w(t);t\in[0,1]\}$, then we have the following relation
\begin{equation*}
	\begin{aligned}
		&{P}\{\gamma \ {\rm has} \ {\rm at} \ {\rm least}  \ {\rm type} \ \alpha(1),\cdots,\alpha(n) \ {\rm up} \ {\rm to} \ \varepsilon\}\\
		\geq&{P}\{w(\frac{1}{n})\in{B}(Z_{\alpha_{1}},\varepsilon),\cdots,w(\frac{n}{n})\in{B}(Z_{\alpha_{n}},\varepsilon)\}.
	\end{aligned}
\end{equation*}
If $w(\frac{1}{n})\in{B}(Z_{\alpha_{1}},\frac{\varepsilon}{n})$ it means that
\begin{equation*}
	\begin{aligned}
		\|w(\frac{1}{n})-Z_{\alpha_{1}}\|\leq\frac{\varepsilon}{n}.
	\end{aligned}
\end{equation*}
If also $w(\frac{2}{n})-w(\frac{1}{n})\in{B}(Z_{\alpha_{2}}-Z_{\alpha_{1}},\frac{\varepsilon}{n})$, which implies that
\begin{equation*}
	\begin{aligned}
		\|w(\frac{2}{n})-Z_{\alpha_{2}}\| \leq\frac{2\varepsilon}{n}.
	\end{aligned}
\end{equation*}
By means of this induction process, we can derive the following result,
\begin{equation*}
\begin{aligned}
&{P}\{w(\frac{1}{n})\in{B}(Z_{\alpha_{1}},\varepsilon),\cdots,w(\frac{n}{n})\in{B}(Z_{\alpha_{n}},\varepsilon)\}\\
\geq&{P}\{w(\frac{1}{n})\in{B}(Z_{\alpha_{1}},\frac{\varepsilon}{n}),w(\frac{2}{n})-w(\frac{1}{n})\in{B}(Z_{\alpha_{2}}-Z_{\alpha_{1}},\frac{\varepsilon}{n}),\cdots,\\
&w(\frac{n}{n})-w(\frac{n-1}{n})\in{B}(Z_{\alpha_{n}}-Z_{\alpha_{n-1}},\frac{\varepsilon}{n})\}\\
=&\Pi_{k=1}^{n}{P}\{w(\frac{1}{n})\in{B}(Z_{\alpha_{k}}-Z_{\alpha_{k-1}},\frac{\varepsilon}{n})\}.
\end{aligned}
\end{equation*}
Denote the density of $w(\frac{1}{n})$ by $p$, and $p$ is decreasing in when $\|x\|$ increases, so we have
\begin{equation}\label{Th2.1.1}
	\begin{aligned}
		&\Pi_{k=1}^{n}{P}\{w(\frac{1}{n})\in{B}(Z_{\alpha_{k}}-Z_{\alpha_{k-1}},\frac{\varepsilon}{n})\}\\
		=&\Pi_{k=1}^{n}\int_{{B}(Z_{\alpha_{k}}-Z_{\alpha_{k-1}},\frac{\varepsilon}{n})}p(x)dx\\
		\geq&\big{(}\int_{{B}(Z_{\alpha_{j\ast}}-Z_{\alpha_{j\ast-1}},\frac{\varepsilon}{n})}p(x)dx\big{)}^{n}\\
		=&P^{n}\{w(\frac{1}{n})\in{B}(Z_{\alpha_{j\ast}}-Z_{\alpha_{j\ast-1}},\frac{\varepsilon}{n})\}\\
		=&P^{n}\{w(1)\in{B}(\sqrt{n}Z_{\alpha_{j\ast}}-\sqrt{n}Z_{\alpha_{j\ast-1}},\frac{\varepsilon}{\sqrt{n}})\}.
	\end{aligned}
\end{equation}
Next, we consider the probability
\begin{equation}\label{Th2.1.3}
	\begin{aligned}
		&P\{w(1)\in{B}(\sqrt{n}Z_{\alpha_{j\ast}}-\sqrt{n}Z_{\alpha_{j\ast-1}},\frac{\varepsilon}{\sqrt{n}})\}\\
		=&\int_{{B}(\sqrt{n}Z_{\alpha_{j\ast}}-\sqrt{n}Z_{\alpha_{j\ast-1}},\frac{\varepsilon}{\sqrt{n}})}(\frac{1}{2\pi})^{\frac{d}{2}} \exp\{-\frac{\|x\|^{2}}{2}\}dx,
		\end{aligned}
	\end{equation}
by the change of variable $y=\frac{\sqrt{n}}{\varepsilon}(x-\sqrt{n}Z_{\alpha_{j\ast}}+\sqrt{n}Z_{\alpha_{j\ast-1}})$, we rewrite (\ref{Th2.1.3}) as follows:
\begin{equation*}
	\begin{aligned}
		=&\int_{{B}(0,1)}(\frac{\varepsilon}{\sqrt{n}})^{d}(\frac{1}{2\pi})^{\frac{d}{2}} \exp\{-\frac{\|\frac{y\varepsilon}{\sqrt{n}}+\sqrt{n}Z_{\alpha_{j\ast}}-\sqrt{n}Z_{\alpha_{j\ast-1}}\|^{2}}{2}\}dy.
		\end{aligned}
\end{equation*}
Next, by the relation
\begin{equation*}
	\begin{aligned}
		&\frac{(\frac{\varepsilon}{\sqrt{n}})^{d}(\frac{1}{2\pi})^{\frac{d}{2}} \exp\{-\frac{\|\frac{y\varepsilon}{\sqrt{n}}+\sqrt{n}Z_{\alpha_{j\ast}}-\sqrt{n}Z_{\alpha_{j\ast-1}}\|^{2}}{2}\}}{(\frac{1}{2\pi})^{\frac{d}{2}}\exp\{-\frac{n\|Z_{\alpha_{j\ast}}-Z_{\alpha_{j\ast-1}}\|^{2}}{2}\} {\rm vol}({B}(\sqrt{n}Z_{\alpha_{j\ast}}-\sqrt{n}Z_{\alpha_{j\ast-1}},\frac{\varepsilon}{\sqrt{n}}))}\\
		=&\frac{\exp\{\frac{n\|Z_{\alpha_{j\ast}}-Z_{\alpha_{j\ast-1}}\|^{2}-\|\frac{y\varepsilon}{\sqrt{n}}+\sqrt{n}Z_{\alpha_{j\ast}}-\sqrt{n}Z_{\alpha_{j\ast-1}}\|^{2}}{2}\}}{{\rm vol}({B}(0,1))}\\
		\leq& \frac{\exp\{\|\frac{y\varepsilon}{\sqrt{n}}\|^{2}\}}{{\rm vol}({B}(0,1))}\leq\frac{e^{\varepsilon^{2}}}{{\rm vol}({B}(0,1))},
	\end{aligned}
\end{equation*}
take $\frac{e^{\varepsilon^{2}}}{{\rm vol}({B}(0,1))} $ as the dominating function, then we can apply the dominated convergence theorem
\begin{equation*}
	\begin{aligned}
		&\lim_{n\rightarrow+\infty}\int_{{B}(0,1)}\frac{(\frac{\varepsilon}{\sqrt{n}})^{d}(\frac{1}{2\pi})^{\frac{d}{2}} \exp\{-\frac{\|\frac{y\varepsilon}{\sqrt{n}}+\sqrt{n}Z_{\alpha_{j\ast}}-\sqrt{n}Z_{\alpha_{j\ast-1}}\|^{2}}{2}\}}{(\frac{1}{2\pi})^{\frac{d}{2}}\exp\{-\frac{n\|Z_{\alpha_{j\ast}}-Z_{\alpha_{j\ast-1}}\|^{2}}{2}\} {\rm vol}({B}(\sqrt{n}Z_{\alpha_{j\ast}}-\sqrt{n}Z_{\alpha_{j\ast-1}},\frac{\varepsilon}{\sqrt{n}}))}dy\\
		=&\int_{\|y\| \leq 1}\frac{\sum_{i=1}^{d}2\varepsilon z_{i}y_{i}}{{\rm vol}({B}(0,1))}dy_{1}\cdots dy_{n}=c_{\varepsilon},
	\end{aligned}
\end{equation*}
where $Z_{\alpha_{j\ast}}-Z_{\alpha_{j\ast-1}}=(z_{1},\cdots,z_{d})$, $y=(y_{1},\cdots,y_{d})$. Thus we have the following relation
\begin{equation}\label{Th2.1.2}
	\begin{aligned}
		&P\{w(1)\in{B}(\sqrt{n}Z_{\alpha_{j\ast}}-\sqrt{n}Z_{\alpha_{j\ast-1}},\frac{\varepsilon}{\sqrt{n}})\}\\
		\sim&c_{\varepsilon}(\frac{1}{2\pi})^{\frac{d}{2}}\exp\{-\frac{n\|Z_{\alpha_{j\ast}}-Z_{\alpha_{j\ast-1}}\|^{2}}{2}\} {\rm vol}({B}(\sqrt{n}Z_{\alpha_{j\ast}}-\sqrt{n}Z_{\alpha_{j\ast-1}},\frac{\varepsilon}{\sqrt{n}})), \ n\rightarrow+\infty.
	\end{aligned}
\end{equation}
Denote $C_{\varepsilon}=\frac{c_{\varepsilon}\varepsilon^{d}}{2^{d/2}\Gamma(\frac{d}{2}+1)}$, we can get that, there exists a $n_{0}$, for any $n>n_{0}$
\begin{equation*}
	\begin{aligned}
&{P}\{\gamma \ {\rm has} \ {\rm at} \ {\rm least}  \ {\rm type} \ \alpha(1),\cdots,\alpha(n) \ {\rm up} \ {\rm to} \ \varepsilon\}\\
\geq& P^{n}\{w(1)\in{B}(\sqrt{n}Z_{\alpha_{j\ast}}-\sqrt{n}Z_{\alpha_{j\ast-1}},\frac{\varepsilon}{\sqrt{n}})\}\\
\geq&{\bigg(}\frac{C_{\varepsilon}}{n^{d/2}}\exp\{-\frac{n\|Z_{\alpha_{j\ast}}-Z_{\alpha_{j\ast-1}}\|^{2}}{2}\}{\bigg)}^{n}.
\end{aligned}
\end{equation*}
Then we take logarithm for both side,
\begin{equation*}
	\begin{aligned}
		&n^{-1} \ln P\{\gamma \ {\rm has} \ {\rm at} \ {\rm least}  \ {\rm type} \ \alpha(1),\cdots,\alpha(n) \ {\rm up} \ {\rm to} \ \varepsilon\}\\
		\geq & \ln C_{\varepsilon}+\ln n^{-d/2}-\frac{n\|Z_{\alpha_{j\ast}}-Z_{\alpha_{j\ast-1}}\|^{2}}{2}.
	\end{aligned}
\end{equation*}
Finally, we can get
\begin{equation*}
	\begin{aligned}
		&\lim_{n\rightarrow+\infty}n^{-2}\ln{P}\{\gamma \ {\rm has} \ {\rm at} \ {\rm least}  \ {\rm type} \ \alpha(1),\cdots,\alpha(n) \ {\rm up} \ {\rm to} \ \varepsilon\}\\
		\geq&-\frac{\|Z_{\alpha_{j\ast}}-Z_{\alpha_{j\ast-1}}\|}{2}.
	\end{aligned}
\end{equation*}
\end{proof}

Next, we consider the asymptotic behavior of probability that polygonal line can be almost inscribed into Brownian trajectory. 

\begin{theorem}
	Let $\{W(t);t\in[0,1]\}$ be a Brownian motion in dimension $d\geq3$. Given that $x$ is the starting point of the Brownian motion, let $z_{1},z_{2},\cdots,z_{n}\neq{x}$ be distinct fixed points such that $z_{k}\neq{z_{j}}$ for $k\neq{j}$ then
	\begin{equation*}
		\begin{aligned}
			&P_{x}\{\exists \ 0\leq{t}_{1}<\cdots<t_{n}\leq{1}:W(t_{1})\in{B(z_{1},\varepsilon)},W(t_{2})\in{B(z_{2},\varepsilon)},\cdots,W(t_{n})\in{B(z_{n},\varepsilon)}\}\\
			\asymp&{\|x-z_{1}\|}^{2-d}{\|z_{1}-z_{2}\|}^{2-d}\cdots{\|z_{n-1}-z_{n}\|}^{2-d}(\varepsilon^{d-2})^{n}, \quad \varepsilon\rightarrow0_{+}.
		\end{aligned}
	\end{equation*}
\end{theorem}

\begin{proof}
	First consider the simple case
	\begin{equation}\label{Tm2.2.5}
		\begin{aligned} P_{x}\{\exists{t}\in[0,1]:W(t)\in{B(z,\varepsilon)}\}.
	\end{aligned}
\end{equation}
	To deal with this probability, we have
	\begin{equation*}
		\begin{aligned}
			P_{x}\{\exists \ {t}\in[0,1]:W(t)\in{B(z,\varepsilon)}\}
			=&P_{\frac{x}{\varepsilon}}\{\exists \ {t}\in[0,1]:\frac{W(t)}{\varepsilon}\in{B(\frac{z}{\varepsilon},1)}\}\\
			=&P_{\frac{x}{\varepsilon}-\frac{z}{\varepsilon}}\{\exists{t}\in[0,1]:\frac{W(t)}{\varepsilon}\in{B(0,1)}\}.
		\end{aligned}
	\end{equation*}
	By the scaling property
	$$\frac{W(t)}{\varepsilon}\overset{d}{=}W(\frac{t}{\varepsilon^{2}}), \quad t\geq0,$$
	we can get
	\begin{equation*}
		\begin{aligned}
			P_{\frac{x}{\varepsilon}-\frac{z}{\varepsilon}}\{\exists{t}\in[0,1]:\frac{W(t)}{\varepsilon}\in{B(0,1)}\}=P_{\frac{x}{\varepsilon}-\frac{z}{\varepsilon}}\{\exists{t}\in[0,\frac{1}{\varepsilon^{2}}]:W(t)\in{B(0,1)}\}.
		\end{aligned}
	\end{equation*}
	Consider these events
	\begin{equation*}
		\begin{aligned}
			A_{1}&=\{\exists{t}\geq{0}:W(t)\in{B(0,1)}\},\\
			A_{2}&=\{\exists{t}\in[\frac{1}{\varepsilon^{2}},+\infty):W(t)\in{B(0,1)}\},\\
			A_{3}&=\{\exists{t}\in[0,\frac{1}{\varepsilon^{2}}]:W(t)\in{B(0,1)}\},
		\end{aligned}
	\end{equation*}
	because $A_{1}=A_{2}\cup{A_{3}}$, and $P\{A_{2}\cup{A_{3}}\}\leq{P\{A_{2}\}}+P\{A_{3}\}$. So
	\begin{equation}\label{Tm2.2.1}
		\begin{aligned}
			{P}_{\frac{x}{\varepsilon}-\frac{z}{\varepsilon}}\{\exists{t}\geq{0}:W(t)\in{B(0,1)}\}
			\geq{P}_{\frac{x}{\varepsilon}-\frac{z}{\varepsilon}}\{\exists{t}\in[0,\frac{1}{\varepsilon^{2}}]:W(t)\in{B(0,1)}\},
	\end{aligned}
\end{equation}
and
\begin{equation}\label{Tm2.2.2}
	\begin{aligned}
		&{P}_{\frac{x}{\varepsilon}-\frac{z}{\varepsilon}}\{\exists{t}\in[0,\frac{1}{\varepsilon^{2}}]:W(t)\in{B(0,1)}\}\\
		\geq&{P}_{\frac{x}{\varepsilon}-\frac{z}{\varepsilon}}\{\exists{t}\geq{0}:W(t)\in{B(0,1)}\}-P_{\frac{x}{\varepsilon}-\frac{z}{\varepsilon}}\{\exists{t}\in[\frac{1}{\varepsilon^{2}},+\infty]:W(t)\in{B(0,1)}\}.
		\end{aligned}
	\end{equation}
	By the known result in \cite{22}
	\begin{equation*}
		\begin{aligned}
			{P}_{\frac{x}{\varepsilon}-\frac{z}{\varepsilon}}\{\exists{t}\geq{0}:W(t)\in{B(0,1)}\}\sim\frac{\Gamma(d/2-1)}{2\pi^{d/2}}{\|x-z\|}^{2-d}\varepsilon^{d-2}, \quad \varepsilon\rightarrow0_{+}.
		\end{aligned}
	\end{equation*}
	Then, from (\ref{Tm2.2.1}), there exists a $\varepsilon_{0}$, for any $\varepsilon<\varepsilon_{0}$,
	\begin{equation}\label{3.1}
		\begin{aligned}
			{P}_{\frac{x}{\varepsilon}-\frac{z}{\varepsilon}}\{\exists{t}\in[0,\frac{1}{\varepsilon^{2}}]:W(t)\in{B(0,1)}\}\leq{C(d)}{\|x-z\|}^{2-d}\varepsilon^{d-2},
		\end{aligned}
	\end{equation}
    which shows the upper bound of the probability (\ref{Tm2.2.5}). Next, we check its lower bound, 
	\begin{equation*}
		\begin{aligned}
			&P_{\frac{x}{\varepsilon}-\frac{z}{\varepsilon}}\{\exists{t}\in[\frac{1}{\varepsilon^{2}},+\infty]:W(t)\in{B(0,1)}\}\\
			=&P\{\exists{t}\in[\frac{1}{\varepsilon^{2}},+\infty]:\frac{x}{\varepsilon}-\frac{z}{\varepsilon}+W(t)\in{B(0,1)}\}\\
			=&P\{\exists{s}\geq{0}:\frac{x}{\varepsilon}-\frac{z}{\varepsilon}+W(\frac{1}{\varepsilon^{2}})+[W(s+\frac{1}{\varepsilon^{2}})-W(\frac{1}{\varepsilon^{2}})]\in{B}(0,1)\}.
		\end{aligned}
	\end{equation*}
	By the Markov property, $\{W(s); s\geq{0}\}$ is a Brownian motion, then the process $\{W(s+\frac{1}{\varepsilon^{2}})-W(\frac{1}{\varepsilon^{2}}); s\geq{0}\}$ is again a Brownian motion started from the origin. Denote $\xi=\frac{x}{\varepsilon}-\frac{z}{\varepsilon}+W(\frac{1}{\varepsilon^{2}})$, $p_{\xi}(u)$ is the density of random variable $\xi$, then
	\begin{equation*}
		\begin{aligned}
			&P\{\exists{s}\geq{0}:\frac{x}{\varepsilon}-\frac{z}{\varepsilon}+W(\frac{1}{\varepsilon^{2}})+[W(s+\frac{1}{\varepsilon^{2}})-W(\frac{1}{\varepsilon^{2}})]\in{B}(0,1)\}\\
			=&P\{\exists{t\geq{0}}:\xi+W(t)\in{B(0,1)}\}\\
			=&\int_{\R^{d}}p_{\xi}(u)P\{\exists{t\geq{0}}:u+W(t)\in{B(0,1)}\}du.
		\end{aligned}
	\end{equation*}
	By the last exit formula,
	\begin{equation*}
		\begin{aligned}
			&\int_{\R^{d}}p_{\xi}(u)P\{\exists{t\geq{0}}:u+W(t)\in{B(0,1)}\}du\\
			=&\int_{\R^{d}}p_{\xi}(u)du\int_{S(0,1)}G(u,v)\delta(dv)\\
			=&\int_{S(0,1)}\delta(dv)\int_{\R^{d}}p_{\xi}(u)G(u,v)du,
		\end{aligned}
	\end{equation*}
		where $\delta$ is the equilibrium measure \cite{22} (i.e. the uniform distribution on $S(0,1)$), $G(u,v)$ is the Green's function,
		$$G(u,v)=\int_{0}^{+\infty}p(t,u,v)dt,$$
		where $p(t,u,v)$ is the transition density of $W$. Then
	\begin{equation*}
		\begin{aligned}
			\int_{\R^{d}}p_{\xi}(u)G(u,v)du
			=&\int_{\R^{d}}p_{\xi}(u)du\int_{0}^{+\infty}p(t,u,v)dt\\
			=&\int_{0}^{+\infty}dt\int_{\R^{d}}p_{\xi}(u)p(t,u,v)du.
		\end{aligned}
	\end{equation*}
Since the second integral is a convolution, we define a new random variable $\beta_{t}\sim{N(\frac{x}{\varepsilon}-\frac{z}{\varepsilon},t+\frac{1}{\varepsilon^{2}})}$, then
	\begin{equation*}
		\begin{aligned}
			\int_{0}^{+\infty}dt\int_{\R^{d}}p_{\xi}(u)p(t,u,v)du
			=&\int_{0}^{+\infty}p_{\beta_{t}}(v)dt\\
			=&\int_{\frac{1}{\varepsilon^{2}}}^{+\infty}\frac{1}{(\sqrt{2\pi{t}})^{d}}e^{-\frac{\|v-\frac{x}{\varepsilon}+\frac{z}{\varepsilon}\|^{2}}{2t}}dt.
		\end{aligned}
	\end{equation*}
	By the change of variable $\frac{\|v-\frac{x}{\varepsilon}+\frac{z}{\varepsilon}\|^{2}}{2t}=\frac{1}{s}$, we have
	\begin{equation*}
		\begin{aligned}
			\int_{\frac{1}{\varepsilon^{2}}}^{+\infty}\frac{1}{(\sqrt{2\pi{t}})^{d}}e^{-\frac{\|v-\frac{x}{\varepsilon}+\frac{z}{\varepsilon}\|^{2}}{2t}}dt=\int_{\frac{2}{\|v\varepsilon-x+z\|^{2}}}^{+\infty}\frac{\|v-\frac{x}{\varepsilon}+\frac{z}{\varepsilon}\|^{2-d}}{2(\sqrt{\pi{s}})^{d}}e^{-\frac{1}{s}}ds,
		\end{aligned}
	\end{equation*}
	here $\|v-\frac{x}{\varepsilon}+\frac{z}{\varepsilon}\|^{2-d}={\|x-z-\varepsilon{v}\|}^{2-d}\varepsilon^{d-2}$. Since $\varepsilon$ is sufficiently small, we ignore $\varepsilon{v}$ here. From (\ref{Tm2.2.2}) we have,
	\begin{equation}\label{3.2}
		\begin{aligned}
			&{P}_{\frac{x}{\varepsilon}-\frac{z}{\varepsilon}}\{\exists{t}\in[0,\frac{1}{\varepsilon^{2}}]:W(t)\in{B(0,1)}\}\\
			\geq&{P}_{\frac{x}{\varepsilon}-\frac{z}{\varepsilon}}\{\exists{t}\geq{0}:W(t)\in{B(0,1)}\}-P_{\frac{x}{\varepsilon}-\frac{z}{\varepsilon}}\{\exists{t}\in[\frac{1}{\varepsilon^{2}},+\infty]:W(t)\in{B(0,1)}\}\\
			\geq&\big{(}\frac{\Gamma(d/2-1)}{2\pi^{d/2}}-\int_{\frac{2}{\|z-x\|^{2}}}^{+\infty}\frac{1}{2(\sqrt{\pi{s}})^{d}}e^{-\frac{1}{s}}ds\big{)}{\|x-z\|}^{2-d}\varepsilon^{d-2}\\
			=&\int_{0}^{\frac{2}{\|z-x\|^{2}}}\frac{1}{2(\sqrt{\pi{s}})^{d}}e^{-\frac{1}{s}}ds{\|x-z\|}^{2-d}\varepsilon^{d-2}\\
			=&{c}{\|x-z\|}^{2-d}\varepsilon^{d-2}, \quad \varepsilon\rightarrow0_{+}.
		\end{aligned}
	\end{equation}
	Combine (\ref{3.1}) and (\ref{3.2}), we can get the asymptotic result
	\begin{equation}\label{Tm2.2.6}
		\begin{aligned}
			P_{x}\{\exists{t}\in[0,1]:W(t)\in{B(z,\varepsilon)}\}\asymp{\|x-z\|}^{2-d}\varepsilon^{d-2}.
		\end{aligned}
	\end{equation}
	Consider the case of $n=2$, estimate the following probability
	$$P_{x}\{\exists0\leq{t}_{1}<t_{2}\leq{1}:W(t_{1})\in{B(z_{1},\varepsilon)},W(t_{2})\in{B(z_{2},\varepsilon)}\}.$$
	Denote the stopping times $$\tau_{1}=\min\{{0\leq{t_{1}}\leq1:W(t_{1})\in{B}(z_{1},\varepsilon)}\},$$ $$\tau_{2}=\min\{{\tau_{1}<{t_{2}}\leq1:W(t_{2})\in{B}(z_{2},\varepsilon)}\},$$
	then we can rewrite the probability
	\begin{equation*}
		\begin{aligned}
			&P_{x}\{\exists 0\leq{t}_{1}<t_{2}\leq{1}:W(t_{1})\in{B(z_{1},\varepsilon)},W(t_{2})\in{B(z_{2},\varepsilon)}\}\\
			=&\mathbb{E}(\mathbf{1}_{0\leq{\tau}_{1}<{\tau}_{2}\leq1})\\
			=&\mathbb{E}_{x}(\mathbf{1}_{0\leq{\tau}_{1}<1})\mathbb{E}_{W(\tau_{1})}(\mathbf{1}_{0<{\tau}_{2}\leq{1-\tau_{1}}}).
		\end{aligned}
	\end{equation*}
	The asymptotic result of $\mathbb{E}_{x}(\mathbf{1}_{0\leq{\tau}_{1}<1})$ is exactly (\ref{Tm2.2.6}). Since we consider $\varepsilon\rightarrow{0}$, here we replace $W(\tau_{1})$ by $z_{1}$. By the result (\ref{3.1}), we can get
	\begin{equation}\label{3.3}
		\begin{aligned}
			\mathbb{E}_{z_{1}}(\mathbf{1}_{0<{\tau}_{2}\leq{1-\tau_{1}}})\leq{c_{2}}{\|z_{1}-z_{2}\|}^{2-d}\varepsilon^{d-2}.
		\end{aligned}
	\end{equation}
	From (\ref{3.2}), we have
	\begin{equation}\label{3.4}
		\begin{aligned}
			\mathbb{E}_{z_{1}}(\mathbf{1}_{0<{\tau}_{2}\leq{1-\tau_{1}}})
			\geq&\int_{0}^{\frac{2-2\tau_{1}}{\|z-x\|^{2}}}\frac{1}{2(\sqrt{\pi{s}})^{d}}e^{-\frac{1}{s}}ds{\|z_{1}-z_{2}\|}^{2-d}\varepsilon^{d-2}\\
			\geq&{c_{3}}{\|z_{1}-z_{2}\|}^{2-d}\varepsilon^{d-2}.
		\end{aligned}
	\end{equation}
	From this, we can get the asymptotic result
	\begin{equation*}
		\begin{aligned}
			&P_{x}\{0\leq{t}_{1}<t_{2}\leq{1}:W(t_{1})\in{B(z_{1},\varepsilon)},W(t_{2})\in{B(z_{2},\varepsilon)}\}\\
			\asymp&{\|x-z_{1}\|}^{2-d}{\|z_{1}-z_{2}\|}^{2-d}(\varepsilon^{d-2})^{2}.
		\end{aligned}
	\end{equation*}
By means of this induction process, we can derive the following result,
	\begin{equation*}
		\begin{aligned}
			&P_{x}\{0\leq{t}_{1}<\cdots<t_{n}\leq{1}:B(t_{1})\in{B(z_{1},\varepsilon)},\cdots,B(t_{n})\in{B(z_{n},\varepsilon)}\}\\
			\asymp&{\|x-z_{1}\|}^{2-d}{\|z_{1}-z_{2}\|}^{2-d}\cdots{\|z_{n-1}-z_{n}\|}^{2-d}(\varepsilon^{d-2})^{n}.
		\end{aligned}
	\end{equation*}
	Theorem is proved.
\end{proof}

In \cite{12} , A. A. Dorogovtsev, O. Izyumtseva, and N. Salhi show that if an integrator is represented by an invertible operator, then its self-intersection local time must exist. Subsequently, we investigate the asymptotic behavior of multiple self-intersections under the condition that the polygonal line inscribed in the Brownian trajectory. First, we introduce the definition of the integrator.
\begin{definition}\cite{Dor06}
	A centered Gaussian process $\{\eta(t),t\in[0,1]\}$ is called the integrator if there exists such a positive constant $C>0$ that for arbitrary $n\geq1$, partition $0=t_{0}<\cdots<t_{n}=1$ and real numbers $a_{0},\cdots,a_{n-1}$,
	\begin{equation*}
		\begin{aligned}
			\E(\sum_{k=0}^{n-1}a_{k}(\eta(t_{k+1})-\eta(t_{k})))^{2}\leq C\sum_{k=0}^{n-1}a_{k}^{2}\Delta t_{k}^{2},
		\end{aligned}
	\end{equation*}
	where $\Delta t_{k}:=t_{k+1}-t_{k}$.
\end{definition}

Next, we introduce a lemma that gives a representation of Brownian motion as being divided into two independent integrators.

\begin{lemma}\label{lemma2.1}
Let $\{w(t);t\in[0,1]\}$ be a Brownian motion. Fix $0<s_{1}<\cdots<s_{n}<1$, then there exists a representation
\begin{equation*}
\begin{aligned}
w(t)=y(t)+x(t),
\end{aligned}
\end{equation*}
where $x(t)$ is the polygonal line which connects the vertices $0,w(s_{1}),\cdots,w(s_{n})$. Here $x(t)$, $y(t)$ are both integrators and they are independent of each other.
\end{lemma}

\begin{proof}
Consider the polygonal line $x(t)$. Since each edge of the polygonal line is a straight-line segment, it can be represented as a linear combination of its vertices, so we have
\begin{equation*}
\begin{aligned}
x(t)=\left\{
      \begin{array}{lcl}
         \frac{t}{s_{1}}w(s_{1}), \quad 0\leq t\leq s_{1},\\
         \frac{s_{2}-t}{s_{2}-s_{1}}w(s_{1})+\frac{t-s_{1}}{s_{2}-s_{1}}w(s_{2}),\quad s_{1}< t\leq s_{2},\\
         \cdots \\
         \frac{s_{n}-t}{s_{n}-s_{n-1}}w(s_{n-1})+\frac{t-s_{n-1}}{s_{n}-s_{n-1}}w(s_{n}),\quad s_{n-1}< t\leq s_{n}.\\
      \end{array}
    \right.
\end{aligned}
\end{equation*}
For a partition $0= t_{0}< \cdots < t_{N}=1$, define a sequence of corresponding constants $c_{0},\cdots,c_{N-1}$. If $x(t_{m})$ and $x(t_{m+1})$ lie in different edges of the polygonal line, we add a point $s_{k}$ to the partition between $t_{m}$ and $t_{m+1}$. From this, we build a new partition $0= \widetilde{t}_{0}< \cdots < \widetilde{t}_{q}=1$. Next, take any $\widetilde{c}_{p}$ such that when ${t}_{p}\in(t_{m},s_{m+1}]$, $\widetilde{c}_{p}=c_{m}$. To show that $x$ is an integrator, we estimate:
\begin{equation*}
\begin{aligned}
&\E(\sum_{m=0}^{N-1}c_{m}(x(t_{m+1})-x(t_{m})))^{2}\\
=&\E(\sum_{k=0}^{n-1}\sum_{\widetilde{t}_{p}\in(s_{k_{1}};s_{k_{1}+1}]}\widetilde{c}_{p}(\widetilde{t}_{p+1}-\widetilde{t}_{p})\frac{w(s_{k+1})-w(s_{k})}{s_{k+1}-s_{k}})^{2}\\
=&\E(\sum_{k=0}^{n-1}\frac{w(s_{k+1})-w(s_{k})}{s_{k+1}-s_{k}}\sum_{\widetilde{t}_{p}\in(s_{k_{1}};s_{k_{1}+1}]}\widetilde{c}_{p}(\widetilde{t}_{p+1}-\widetilde{t}_{p}))^{2}\\
=&\sum_{k=0}^{n-1}(\sum_{\widetilde{t}_{p}\in(s_{k_{1}};s_{k_{1}+1}]}\widetilde{c}_{p}(\widetilde{t}_{p+1}-\widetilde{t}_{p}))^{2}\\
\leq&\sum_{k=0}^{n-1}(\sum_{\widetilde{t}_{p}\in(s_{k_{1}};s_{k_{1}+1}]}\widetilde{c}_{p}^{2}\sum_{\widetilde{t}_{p}\in(s_{k_{1}};s_{k_{1}+1}]}(\widetilde{t}_{p+1}-\widetilde{t}_{p})^{2})\\
\leq &\sum_{m=0}^{N-1}c^{2}_{m}\Delta {t}_{m}^{2}.
\end{aligned}
\end{equation*}
From this follows that $x$ is an integrator by the definition of the integrator. Since $w(t)$ and $x(t)$ are both integrator, and $y(t)=w(t)-x(t)$, it follows that $y(t)$ is also an integrator. Next, since $y(t)$ and $x(t)$ are both Gaussian processes, to verify their independence, we only need to check that their covariance is zero. For any $t\in[s_{k-1},s_{k}]$, $s\in[s'_{k-1},s'_{k}]$, the covariance of $y(s)$ and $x(t)$ can be expressed as follows:
\begin{equation}\label{11.11}
\begin{aligned}
Cov(x,y)=& \E[(x(t)-\E (x(t)))(y(s)-\E (y(s)))]\\
        =&\frac{s_{k}-t}{s_{k}-s_{k-1}}s_{k-1}\wedge s+\frac{t-s_{k-1}}{s_{k}-s_{k-1}}s_{k}\wedge s\\
        &-\frac{s_{k}-t}{s_{k}-s_{k-1}}\frac{s'_{k}-s}{s'_{k}-s'_{k-1}}s_{k-1}\wedge s'_{k-1}-\frac{t-s_{k-1}}{s_{k}-s_{k-1}}\frac{s-s'_{k-1}}{s'_{k}-s'_{k-1}}s_{k}\wedge s'_{k}\\
        &- \frac{s_{k}-t}{s_{k}-s_{k-1}}\frac{s-s'_{k-1}}{s'_{k}-s'_{k-1}}s_{k-1}\wedge s'_{k}-\frac{t-s_{k-1}}{s_{k}-s_{k-1}}\frac{s'_{k}-s}{s'_{k}-s'_{k-1}}s_{k}\wedge s'_{k-1}.
\end{aligned}
\end{equation}
Next, we divide the problem into two cases. Firstly, assume $s_{k-1}\leq t\leq s_{k} \leq s'_{k-1}\leq s \leq s'_{k}$, then (\ref{11.11}) can be expressed as follows:
\begin{equation*}
\begin{aligned}
Cov(x,y)=&\frac{(s_{k}-t)s_{k-1}(s'_{k}-s'_{k-1}-s'_{k}+s-s+s'_{k-1})}{(s_{k}-s_{k-1})(s'_{k}-s'_{k-1})}\\
&+\frac{(t-s_{k-1})s_{k}(s'_{k}-s'_{k-1}-s+s'_{k-1}-s'_{k}+s)}{(s_{k}-s_{k-1})(s'_{k}-s'_{k-1})}=0.
\end{aligned}
\end{equation*}
Secondly, assume $ s'_{k-1}\leq s \leq s'_{k} \leq s_{k-1}\leq t\leq s_{k} $, then (\ref{11.11}) can be expressed as follows:
\begin{equation*}
\begin{aligned}
Cov(x,y)=&\frac{(s_{k}-t)[s(s'_{k}-s'_{k-1})-s'_{k-1}(s'_{k}-s)-s'_{k}(s-s'_{k-1})]}{(s_{k}-s_{k-1})(s'_{k}-s'_{k-1})}\\
 &+\frac{(t-s_{k-1})[s(s'_{k}-s'_{k-1})-s'_{k}(s-s'_{k-1})-s'_{k-1}(s'_{k}-s)]}{(s_{k}-s_{k-1})(s'_{k}-s'_{k-1})}=0.
\end{aligned}
\end{equation*}
Thus we can conclude that $y(t)$ and $x(t)$ are independent.
\end{proof}

\begin{theorem}
Let $\{w(t);t\in[0,1]\}$ be a Brownian motion in $\R^{2}$. Fix $0\leq s_{1}< \cdots < s_{n} \leq 1$, then for the conditional expectation of the self-intersection local time of the Brownian motion the following result holds
\begin{equation*}
	\begin{aligned}
		&\E(\int_{0\leq t_{1}< t_{2} \leq1}\delta_{\vec{u}}(w(t_{2})-w(t_{1}))dt_{1}dt_{2}|w(s_{1})=v_{1},\cdots,w(s_{n})=v_{n})\\
		\sim &\frac{(s^{*})^{2}}{\pi u v^{*}}e^{-\sqrt{2}uv^{*}/s^{*}}, \quad v^{*}\rightarrow\infty,
	\end{aligned}
\end{equation*}
where $\vec{u}=(u,u)$, $u\neq0$ is fixed,  $k^{*}=\underset{k}{\arg\min}\|v_{k-1}+v_{k}\|$, $v^{*}=\|v_{k^{*}-1}+v_{k^{*}}\|$, $s^{*}=s_{k^{*}}-s_{k^{*}-1}$.
\end{theorem}

\begin{proof}
From Lemma \ref{lemma2.1}, the conditional expectation of self-intersection local time can be written as follows:
\begin{equation}\label{TM2.3.13}
\begin{aligned}
&\E(\int_{0\leq t_{1}< t_{2} \leq1}\delta_{0}(w(t_{2})-w(t_{1}))dt_{1}dt_{2}|w(s_{1})=v_{1},\cdots,w(s_{n})=v_{n})\\
=&\E(\int_{0\leq t_{1}< t_{2} \leq1}\delta_{0}(\Delta y(t_{2},t_{1})+\Delta x(t_{2},t_{1}))dt_{1}dt_{2}|w(s_{1})=v_{1},\cdots,w(s_{n})=v_{n})\\
=&\E(\int_{0\leq t_{1}< t_{2} \leq1}\delta_{0}(\Delta y(t_{2},t_{1})+\Delta x(t_{2},t_{1}))dt_{1}dt_{2})|_{x(s_{1})=v_{1},\cdots,x(s_{n})=v_{n}}.
\end{aligned}
\end{equation}
First, we consider
\begin{equation*}
\begin{aligned}
\E(\delta_{0}(\Delta y(t_{2},t_{1})+\Delta x(t_{2},t_{1})))=q(\Delta x(t_{2},t_{1})),
\end{aligned}
\end{equation*}
where $q$ is the density of $\Delta y$, and $q$ is uniformly integrable when $0\leq t_{1}< t_{2} \leq1$. When $t\in[s_{k-1},s_{k}]$, denote $v_{k}=(v_{k}^{1},v_{k}^{2})$.
Now, the covariance matrix of $y(t)$ is
\begin{equation*}
	\begin{aligned}
		\left\{
		\begin{array}{cc}
			g(t) & 0\\
			0 & g(t)\\
		\end{array}
		\right\},
	\end{aligned}
\end{equation*}
where $g(t)=t-\frac{t^{2}}{s_{k}-s_{k-1}}$. Then we can rewrite (\ref{TM2.3.13}) as follows:
\begin{equation*}
	\begin{aligned}
		&\E(\int_{0\leq t_{1}< t_{2} \leq1}\delta_{0}(\Delta y(t_{2},t_{1})+\Delta x(t_{2},t_{1}))dt_{1}dt_{2})|_{x(s_{1})=v_{1},\cdots,x(s_{n})=v_{n}}\\
		=&\sum_{k=2}^{n}\int_{s_{k-1}\leq t_{1}< t_{2} \leq s_{k}}q(\Delta x(t_{2},t_{1}))dt_{1}dt_{2}\\
		&+\sum_{i\neq j}\int_{s_{i-1}\leq t_{1} \leq s_{i}}\int_{s_{j-1}\leq t_{2} \leq s_{j}}q(\Delta x(t_{2},t_{1}))dt_{1}dt_{2}.
	\end{aligned}
\end{equation*}
Thus, when $t_{1}, t_{2}$ in the same interval $[s_{k-1},s_{k}]$, we can get the covariance of $\Delta y$: 
\begin{equation*}
	\begin{aligned}
		g(t_{2}-t_{1})=(t_{2}-t_{1})-\frac{(t_{2}-t_{1})^{2}}{s_{k}-s_{k-1}}.
	\end{aligned}
\end{equation*}
The conditional expectation can be written as follows:
\begin{equation}\label{Tm2.3.1}
	\begin{aligned}
		&\int_{s_{k-1}\leq t_{1}< t_{2} \leq s_{k}}\frac{s_{k}-s_{k-1}}{2\pi (t_{2}-t_{1})(s_{k}-s_{k-1}-(t_{2}-t_{1}))}\\
		&\exp\{-\frac{\sum_{i=1}^{2}(s_{k}-s_{k-1})(u-\frac{(t_{2}-t_{1})(v^{i}_{k}+v^{i}_{k-1})}{s_{k}-s_{k-1}})^{2}}{2 (t_{2}-t_{1})(s_{k}-s_{k-1}-(t_{2}-t_{1}))}\}dt_{1}dt_{2}.
	\end{aligned}
\end{equation}
Denote $s_{k}-s_{k-1}=s$, by the change of variables $a_{1}=t_{2}-t_{1}, a_{2}=t_{2}+t_{1}$, (\ref{Tm2.3.1}) can be written as follows:
\begin{equation}\label{Tm2.3.2}
	\begin{aligned}
		&\int_{0}^{s}\frac{s}{\pi a_{1}}\exp\{-\frac{\sum_{i=1}^{2}s(u-\frac{a_{1}(v^{i}_{k}+v^{i}_{k-1})}{s})^{2}}{2a_{1}(s-a_{1})}\}da_{1}.
	\end{aligned}
\end{equation}
Denote $\sum_{i=1}^{2}(v^{i}_{k}+v^{i}_{k-1})^{2}=\|v_{k}+v_{k-1}\|^{2}$, $v^{1}_{k}+v^{1}_{k-1}+v^{2}_{k}+v^{2}_{k-1}=v$. Do the change of variable $b=p(a_{1})=\frac{\sum_{i=1}^{2}s(u-\frac{a_{1}(v^{i}_{k}+v^{i}_{k-1})}{s})^{2}}{2a_{1}(s-a_{1})}$. Here the function $p$ has only one minimal value on $[0,s]$. When $a_{1}=s$, $b=\infty$, when $a_{1}=0$, $b=\infty$. To find the minimal value of $p(a_{1})$, we need to check the derivative of the function $p(a_{1})$. By calculation, the extreme point is
\begin{equation*}
	\begin{aligned}
		a^{*}=\frac{-8u^{2}s+4us\sqrt{4u^{2}+(2\|v_{k}+v_{k-1}\|^{2}-4uv)}}{4\|v_{k}+v_{k-1}\|^{2}-8uv},
	\end{aligned}
\end{equation*}
then the minimal value of $p$ is 
\begin{equation*}
	\begin{aligned}
		&f(\|v_{k}+v_{k-1}\|,v)\\
		=&\frac{[16u^{4}s+2u^{2}s(2\|v_{k}+v_{k-1}\|^{2}-4uv)+8u^{3}s\sqrt{4u^{2}+(2\|v_{k}+v_{k-1}\|^{2}-4uv)}]\|v_{k}+v_{k-1}\|^{2}}{q(\|v_{k}+v_{k-1}\|,v)}\\
		&+\frac{u^{2}s(2\|v_{k}+v_{k-1}\|^{2}-4uv)^{2}+4u^{3}sv(2\|v_{k}+v_{k-1}\|^{2}-4uv)}{q(\|v_{k}+v_{k-1}\|,v)}\\
		&+\frac{2u^{2}sv(2\|v_{k}+v_{k-1}\|^{2}-4uv)\sqrt{4u^{2}+(2\|v_{k}+v_{k-1}\|^{2}-4uv)}}{q(\|v_{k}+v_{k-1}\|,v)},
	\end{aligned}
\end{equation*}
where
\begin{equation*}
	\begin{aligned}
		&q(\|v_{k}+v_{k-1}\|,v)\\=&2us^{2}(2\|v_{k}+v_{k-1}\|^{2}-4uv)\sqrt{4u^{2}+(2\|v_{k}+v_{k-1}\|^{2}-4uv)}-8u^{2}s^{2}(2\|v_{k}+v_{k-1}\|^{2}-4uv)\\
		&-32u^{4}s^{2}+16u^{3}s^{2}\sqrt{4u^{2}+(2\|v_{k}+v_{k-1}\|^{2}-4uv)}-8u^{2}s^{2}(2\|v_{k}+v_{k-1}\|^{2}-4uv).
	\end{aligned}
\end{equation*}
When $0<a_{1}<s$ and $u\neq 0$, for function $p(a_{1})=\frac{\sum_{i=1}^{2}s(u-\frac{a_{1}(v^{i}_{k}+v^{i}_{k-1})}{s})^{2}}{2a_{1}(s-a_{1})}$, both $2a_{1}(s-a_{1})$ and $\sum_{i=1}^{2}s(u-\frac{a_{1}(v^{i}_{k}+v^{i}_{k-1})}{s})^{2}$ are greater than $0$. Therefore, the function $f(\|v_{k}+v_{k-1}\|,v)>0$. Here, \( p \) is decreasing on \([0, a^{*}]\) and increasing on \((a^{*}, s]\). Therefore, there exists \( h_{1} \), which is the inverse function of \( p \) on the interval \([0, a^{*}]\), and \( h_{2} \), which is the inverse function of \( p \) on the interval \((a^{*}, s]\), 
\begin{equation*}
	\begin{aligned}
		a_{1}=&h_{1}(b,\|v_{k}+v_{k-1}\|,v)\\
		=&\frac{s^{2}uv+bs^{2}}{\|v_{k}+v_{k-1}\|^{2}+2bs}+\sqrt{(\frac{s^{2}uv+bs^{2}}{\|v_{k}+v_{k-1}\|^{2}+2bs})^{2}-\frac{2u^{2}s^{2}}{\|v_{k}+v_{k-1}\|^{2}+2bs}},
	\end{aligned}
\end{equation*}
and
\begin{equation*}
	\begin{aligned}
		a_{1}=&h_{2}(b,\|v_{k}+v_{k-1}\|,v)\\
		=&\frac{s^{2}uv+bs^{2}}{\|v_{k}+v_{k-1}\|^{2}+2bs}-\sqrt{(\frac{s^{2}uv+bs^{2}}{\|v_{k}+v_{k-1}\|^{2}+2bs})^{2}-\frac{2u^{2}s^{2}}{\|v_{k}+v_{k-1}\|^{2}+2bs}}.
	\end{aligned}
\end{equation*}
Then the integral (\ref{Tm2.3.2}) can be written as follows:
\begin{equation}\label{Tm2.3.3}
	\begin{aligned}
		&\int_{f(\|v_{k}+v_{k-1}\|,v)}^{\infty}\frac{sh_{1}'(b,\|v_{k}+v_{k-1}\|,v)}{\pi h_{1}(b,\|v_{k}+v_{k-1}\|,v)}e^{-b}db\\
		&+\int_{f(\|v_{k}+v_{k-1}\|,v)}^{\infty}\frac{sh_{2}'(b,\|v_{k}+v_{k-1}\|,v)}{\pi h_{2}(b,\|v_{k}+v_{k-1}\|,v)}e^{-b}db,
	\end{aligned}
\end{equation}
where $h_{1}',h_{2}'$ are the derivative of $h_{1},h_{2}$ with respect to $b$
\begin{equation*}
	\begin{aligned}
		&h'_{1}(b,\|v_{k}+v_{k-1}\|,v)\\
		=&\frac{s^{2}(\|v_{k}+v_{k-1}\|^{2}+2bs)-(s^{2}uv+bs^{2})2s}{(\|v_{k}+v_{k-1}\|^{2}+2bs)^{2}}\\
		&+\frac{1}{2\sqrt{(\frac{s^{2}uv+bs^{2}}{\|v_{k}+v_{k-1}\|^{2}+2bs})^{2}-\frac{2u^{2}s^{2}}{\|v_{k}+v_{k-1}\|^{2}+2bs}}}\\
		&[2(\frac{s^{2}uv+bs^{2}}{\|v_{k}+v_{k-1}\|^{2}+2bs})\frac{s^{2}(\|v_{k}+v_{k-1}\|^{2}+2bs)-(s^{2}uv+bs^{2})2s}{(\|v_{k}+v_{k-1}\|^{2}+2bs)^{2}}\\
		&+\frac{4u^{2}s^{3}}{(\|v_{k}+v_{k-1}\|^{2}+2bs)^{2}}].
	\end{aligned}
\end{equation*}
Because the asymptotic result for both integral in (\ref{Tm2.3.3}) are the same, so it is enough to check the first integral of (\ref{Tm2.3.3}),
\begin{equation*}
	\begin{aligned}
		&\int_{f(\|v_{k}+v_{k-1}\|,v)}^{\infty}\frac{sh_{1}'(b,\|v_{k}+v_{k-1}\|,v)}{\pi h_{1}(b,\|v_{k}+v_{k-1}\|,v)}e^{-b}db\\
		=&e^{-f(\|v_{k}+v_{k-1}\|,v)}\int_{0}^{\infty}\frac{sh_{1}'(a+f(\|v_{k}+v_{k-1}\|,v),\|v_{k}+v_{k-1}\|,v)}{\pi h_{1}(a+f(\|v_{k}+v_{k-1}\|,v),\|v_{k}+v_{k-1}\|,v)}e^{-a}da.
	\end{aligned}
\end{equation*}
For  sufficiently large $\|v_{k}+v_{k-1}\|$, one can get the following estimation  
\begin{equation*}
	\begin{aligned}
		|\frac{sh_{1}'(a+f(\|v_{k}+v_{k-1}\|,v),\|v_{k}+v_{k-1}\|,v)\|v_{k}+v_{k-1}\|}{\pi h_{1}(a+f(\|v_{k}+v_{k-1}\|,v),\|v_{k}+v_{k-1}\|,v)}e^{-a}|\leq C_{1}e^{-a},
	\end{aligned}
\end{equation*}
where $C_{1}$ is a constant. Here, we will consider the asymptotic behavior when $\|v_{k}+v_{k-1}\|\rightarrow\infty$, so $f(\|v_{k}+v_{k-1}\|,v)$ is equivaulent to $\frac{\sqrt{2}u}{s}\|v_{k}+v_{k-1}\|$,  $h_{1}(a+f(\|v_{k}+v_{k-1}\|,v),\|v_{k}+v_{k-1}\|,v)$ is equivalent to $\sqrt{2}us\|v_{k}+v_{k-1}\|^{-1}$, $h_{1}'(a+f(\|v_{k}+v_{k-1}\|,v),\|v_{k}+v_{k-1}\|,v)$ is equivalent to $s^{2}\|v_{k}+v_{k-1}\|^{-2}$. Taking $C_{1}e^{-a}$ as the dominating function, we can apply the dominated convergence theorem, 
\begin{equation*}
	\begin{aligned}
		&\lim_{\|v_{k}+v_{k-1}\|\rightarrow\infty}
		\int_{0}^{\infty}\frac{sh_{1}'(a+f(\|v_{k}+v_{k-1}\|,v),\|v_{k}+v_{k-1}\|,v)\|v_{k}+v_{k-1}\|}{\pi h_{1}(a+f(\|v_{k}+v_{k-1}\|,v),\|v_{k}+v_{k-1}\|,v)}e^{-a}da\\
		=&\int_{0}^{\infty}\frac{s^{2}}{\pi u}e^{-a}da=\frac{s^{2}}{\pi u}.
	\end{aligned}
\end{equation*}
Thus we have for $t_{1},t_{2}$ in the same interval
\begin{equation*}
	\begin{aligned}
		&\int_{s_{k-1}\leq t_{1}< t_{2} \leq s_{k}}\frac{s}{2\pi (t_{2}-t_{1})(s-(t_{2}-t_{1}))}\\
		&\exp\{-\frac{\sum_{i=1}^{2}(s)(u-\frac{(t_{2}-t_{1})(v^{i}_{k}+v^{i}_{k-1})}{s})^{2}}{2 (t_{2}-t_{1})(s-(t_{2}-t_{1}))}\}dt_{1}dt_{2}\\
		\sim &\frac{s^{2}}{\pi u}\|v_{k}+v_{k-1}\|^{-1}e^{-\frac{\sqrt{2}u}{s}\|v_{k}+v_{k-1}\|}, \quad \|v_{k}+v_{k-1}\|\rightarrow\infty.
	\end{aligned}
\end{equation*}
When $t_{1}\in[s_{k-1},s_{k}]$ and $t_{2}\in[s_{k'-1},s_{k'}]$, we divide it into two cases. First case is $s_{k-1}\leq t_{1}\leq s_{k} \leq s'_{k-1}\leq t_{2} \leq s'_{k}$, on the second when $ s'_{k-1}\leq t_{2} \leq s'_{k} \leq s_{k-1}\leq t_{1}\leq s_{k} $. We will only deal with the first case, because the analysis of the other case is similar. We denote the variance of $\Delta y$ by $g(t_{2},t_{1})$. By the relation $s'_{k}-s_{k-1}>t_{2}-t_{1}> s'_{k-1}-s_{k}$, there must exist a constant $n>0$, such that $t_{2}=t_{1}+n$. Then regarding the variance $g(t_{2},t_{1})$ as a function of $n$ for a fixed $t_{1}$, we denote it by $k(n)$. Here, $k(n)$ can be expressed as follows:
\begin{equation*}
	\begin{aligned}
		k(n)=an^{2}+bn+c,
	\end{aligned}
\end{equation*}
where $a,b,c$ are some constants related to $t_{1},s_{k},s_{k-1},s'_{k-1},s_{k}$. Because the variance $k(n)\geq 0$, and $k(0)=0$. Based on this property, we can conclude that for $h>0$, the function $k(n)$ is monotonically increasing and strictly greater than $0$. Then, considering the inequality $s'_{k}-s_{k-1}>h> s'_{k-1}-s_{k}$, there must exist two constants $C_{3},C_{4}$ such that $C_{3}\leq k(n) \leq C_{4}$. Now, the conditional expectation can be expressed as follows:
\begin{equation}\label{Tm2.3.4}
	\begin{aligned}
		&\int_{s_{k-1}\leq t_{1}\leq s_{k} \leq s'_{k-1}\leq t_{2} \leq s'_{k}}\frac{1}{2\pi  g(t_{2},t_{1})}\exp\{-\frac{1}{2g(t_{2},t_{1})}\\
		&\sum_{i=1}^{2}(u-\frac{t_{2}(v^{i}_{k'}+v^{i}_{k'-1})}{s'_{k}-s'_{k-1}}+\frac{t_{1}(v^{i}_{k}+v^{i}_{k-1})}{s_{k}-s_{k-1}})^{2}\}dt_{1}dt_{2}\\
		\leq &\int_{s_{k-1}\leq t_{1}\leq s_{k} \leq s'_{k-1}\leq t_{2} \leq s'_{k}}\frac{1}{2\pi  C_{3}}\exp\{-\frac{1}{2C_{4}}\sum_{i=1}^{2}(u-\frac{(t_{2}-t_{1})(v^{i}_{k}+v^{i}_{k-1})}{s_{k}-s_{k-1}})^{2}\}dt_{1}dt_{2}.
	\end{aligned}
\end{equation}
Here, we assume $\frac{(v^{i}_{k}+v^{i}_{k-1})}{s_{k}-s_{k-1}}\leq\frac{(v^{i}_{k'}+v^{i}_{k'-1})}{s'_{k}-s'_{k-1}}$. By the change of variables $a_{1}=t_{2}-t_{1}, a_{2}=t_{2}+t_{1}$, so (\ref{Tm2.3.4}) can be written as follows:
\begin{equation}\label{th2.3.1}
	\begin{aligned}
		&\int_{s'_{k-1}-s_{k}}^{s'_{k}-s_{k-1}}\frac{s'_{k}-s'_{k-1}-a_{1}}{\pi C_{3}}\exp\{-\frac{1}{2C_{4}}\sum_{i=1}^{2}(u-\frac{a_{1}(v^{i}_{k}+v^{i}_{k-1})}{s_{k}-s_{k-1}})^{2}\}da_{1}.
	\end{aligned}
\end{equation}
Denote $\sum_{i=1}^{d}(v^{i}_{k}+v^{i}_{k-1})^{2}=\|v_{k}+v_{k-1}\|^{2}$, by a change of variable $b=\frac{1}{2C_{4}}\sum_{i=1}^{2}(u-\frac{a_{1}(v^{i}_{k}+v^{i}_{k-1})}{s_{k}-s_{k-1}})^{2}$, we can get its inverse function
\begin{equation*}
	\begin{aligned}
		a_{1}=h(b,\|v_{k}+v_{k-1}\|)=\frac{(s_{k}-s_{k-1})(u+\sqrt{2bC_{4}})}{\|v_{k}+v_{k-1}\|},
	\end{aligned}
\end{equation*}
and the derivative of the inverse function is
\begin{equation*}
	\begin{aligned}
		h'(b,\|v_{k}+v_{k-1}\|)=\frac{(s_{k}-s_{k-1})\sqrt{C_{4}}}{\sqrt{2b}\|v_{k}+v_{k-1}\|}.
	\end{aligned}
\end{equation*}
Denote $p(\|v_{k}+v_{k-1}\|)=\frac{(s'_{k-1}-s_{k})^{2}\|v_{k}+v_{k-1}\|^{2}}{2C_{4}(s_{k}-s_{k-1})^{2}}, q(\|v_{k}+v_{k-1}\|)=\frac{(s'_{k}-s_{k-1})^{2}\|v_{k}+v_{k-1}\|^{2}}{2C_{4}(s_{k}-s_{k-1})^{2}}$, the integral (\ref{th2.3.1}) satisfies the following relation,
\begin{equation*}
	\begin{aligned}
		&\int_{s'_{k-1}-s_{k}}^{s'_{k}-s_{k-1}}\frac{s'_{k}-s'_{k-1}-a_{1}}{\pi C_{3}}\exp\{-\frac{1}{2C_{4}}\sum_{i=1}^{2}(u-\frac{a_{1}(v^{i}_{k}+v^{i}_{k-1})}{s_{k}-s_{k-1}})^{2}\}da_{1}\\
		=&\int_{p(\|v_{k}+v_{k-1}\|)}^{q(\|v_{k}+v_{k-1}\|)}\frac{(s'_{k}-s'_{k-1}-h(b))h'(b)}{\pi C_{3}}e^{-b}db\\
		\leq&e^{-p(\|v_{k}+v_{k-1}\|)}\int_{0}^{\infty}\frac{(s'_{k}-s'_{k-1}-h(a+p(\|v_{k}+v_{k-1}\|))h'(a+p(\|v_{k}+v_{k-1}\|))}{\pi C_{3}}e^{-a}da.
	\end{aligned}
\end{equation*}
By the form of $h$ and $h'$, we have
\begin{equation*}
	\begin{aligned}
		h(a+p(\|v_{k}+v_{k-1}\|))=\frac{(s_{k}-s_{k-1})\sqrt{u+2(a+p(\|v_{k}+v_{k-1}\|))C_{3}}}{\|v_{k}+v_{k-1}\|},
	\end{aligned}
\end{equation*}
and
\begin{equation*}
	\begin{aligned}
		h'(a+p(\|v_{k}+v_{k-1}\|))=\frac{(s_{k}-s_{k-1})\sqrt{C_{3}}}{\sqrt{2(a+p(\|v_{k}+v_{k-1}\|))}\|v_{k}+v_{k-1}\|}.
	\end{aligned}
\end{equation*}
Then we have
\begin{equation*}
	\begin{aligned}
		&|\frac{\|v_{k}+v_{k-1}\|^{2}(s'_{k}-s'_{k-1}-h(a+p(\|v_{k}+v_{k-1}\|)))h'(a+p(\|v_{k}+v_{k-1}\|))}{\pi C_{3}}e^{-a}|\\
		\leq &|\frac{\|v_{k}+v_{k-1}\|^{2}(s'_{k}-s'_{k-1})h'(a+\frac{(s'_{k-1}-s_{k})^{2}\|v_{k}+v_{k-1}\|^{2}}{2C_{4}(s_{k}-s_{k-1})^{2}})}{\pi C_{3}}e^{-a}|\\
		=& \frac{\|v_{k}+v_{k-1}\|(s'_{k}-s'_{k-1})(s_{k}-s_{k-1})}{\pi \sqrt{2(a+\frac{(s'_{k-1}-s_{k})^{2}\|v_{k}+v_{k-1}\|^{2}}{2C_{4}(s_{k}-s_{k-1})^{2}})}}e^{-a}\\
		\leq&\frac{(s'_{k}-s'_{k-1})(s_{k}-s_{k-1})^{2}\sqrt{C_{4}}}{\pi (s'_{k-1}-s_{k})}e^{-a}.
	\end{aligned}
\end{equation*}
Here, we will consider the asymptotic behavior when $\|v_{k}+v_{k-1}\|\rightarrow\infty$, so $h(a+p(\|v_{k}+v_{k-1}\|))$ is equivalent to a constat $C_{5}$,  $h'(a+p(\|v_{k}+v_{k-1}\|))$ is equivalent to $C_{6}\|v_{k}+v_{k-1}\|^{-2}$. Taking  $\frac{(s'_{k}-s'_{k-1})(s_{k}-s_{k-1})^{2}\sqrt{C_{4}}}{\pi (s'_{k-1}-s_{k})}e^{-a}$ as the dominating function, we can apply the dominated convergence theorem, 
\begin{equation*}
	\begin{aligned}
		&\lim_{\|v_{k}+v_{k-1}\|\rightarrow\infty}\int_{0}^{\infty}\frac{\|v_{k}+v_{k-1}\|^{2}(s'_{k}-s'_{k-1}-h(a+p))}{\pi C_{3}}h'(a+p)e^{-a}da\\
		=&\int_{0}^{\infty}C_{7}e^{-a}da=C_{7}.
	\end{aligned}
\end{equation*}
So, we can get
\begin{equation*}
	\begin{aligned}
			&\int_{s_{k-1}\leq t_{1}\leq s_{k} \leq s'_{k-1}\leq t_{2} \leq s'_{k}}\frac{1}{2\pi  g(t_{2},t_{1})}\exp\{-\frac{1}{2g(t_{2},t_{1})}\\
		&\sum_{i=1}^{2}(u-\frac{t_{2}(v^{i}_{k'}+v^{i}_{k'-1})}{s'_{k}-s'_{k-1}}+\frac{t_{1}(v^{i}_{k}+v^{i}_{k-1})}{s_{k}-s_{k-1}})^{2}\}dt_{1}dt_{2}\\
		\leq& C_{7}e^{-\frac{(s'_{k-1}-s_{k})^{2}\|v_{k}+v_{k-1}\|^{2}}{2C_{4}(s_{k}-s_{k-1})^{2}}} \|v_{k}+v_{k-1}\|^{-2}, \quad  \|v_{k}+v_{k-1}\|\rightarrow\infty.
	\end{aligned}
\end{equation*}
Thus, we can conclude that when \( t_{1}, t_{2} \) lie in different intervals, they converge to zero faster than when \( t_{1}, t_{2} \) lie in the same interval. Finally, we have the following result,
\begin{equation*}
	\begin{aligned}
		&\E(\int_{0\leq t_{1}< t_{2} \leq1}\delta_{0}(w(t_{2})-w(t_{1}))dt_{1}dt_{2}|w(s_{1})=v_{1},\cdots,w(s_{n})=v_{n})\\
		\sim &\frac{(s^{*})^{2}}{\pi u v^{*}}e^{-\sqrt{2}uv^{*}/s^{*}}, \quad v^{*}\rightarrow\infty,
	\end{aligned}
\end{equation*}
where $k^{*}=\underset{k}{\arg\min}\|v_{k-1}+v_{k}\|$, $v^{*}=\|v_{k^{*}-1}+v_{k^{*}}\|$, $s^{*}=s_{k^{*}}-s_{k^{*}-1}$.
\end{proof}

\section{The random curve driven by the equation with interaction}
In this section, we will consider the motion of a random curve driven by the equation with interaction. First, we introduce the stochastic differential equations with interaction \cite{9}
\begin{equation}\label{2.1}
\begin{aligned}
\left\{
    \begin{array}{lcl}
        d\vec{x}(\vec{u},t)=a(\vec{x}(\vec{u},t),\mu_{t})dt+\int_{\R^{d}}b(\vec{x}(\vec{u},t),\mu_{t},p)w(d\vec{p},dt), \\
        \vec{x}(\vec{u},0)=\vec{u},\quad \vec{u}\in{\R^{d}}, \\
        \mu_{t}=\mu_{0}\circ{\vec{x}(\cdot,t)}^{-1},
      \end{array}
    \right.
\end{aligned}
\end{equation}
where $w(d\vec{p},dt)$ is the Wiener sheet. If one take $\mu_{0}$ as a visitation measure of the curve $\gamma$, the $\mu_{t}$ will be a visitation measure of its image at the moment $t$.
\begin{equation*}
\begin{aligned}
&\mu_{0}(\Delta)=\int_{0}^{1}\mathbf{1}_{\Delta}(\vec{\xi}(s))ds,\\
&\mu_{t}(\Delta)=\int_{0}^{1}\mathbf{1}_{\Delta}(\vec{x}(\vec{\xi}(s),t))ds,\quad \vec{x}(\vec{\xi}(\cdot),t):[0,1]\rightarrow{\R^{d}}.
\end{aligned}
\end{equation*}

The idea of describing the curve's motion using an equation with interaction driven by the curve's visitation measure was proposed in the paper \cite{15} . The coefficients in (\ref{2.1}) satisfy Lipschitz condition with respect to measure-valued argument when the distance between measures is measured by the Wasserstein distance \cite{16} .

\begin{definition}
The Wasserstein distance between the probability measures on $\R^{d}$ is
\begin{equation*}
\begin{aligned}
\rho(\mu,\nu)=\inf_{Q(\mu,\nu)}\int_{\R^{d}}\int_{\R^{d}}\frac{\|\vec{u}-\vec{v}\|}{1+\|\vec{u}-\vec{v}\|}\eta(d\vec{u},d\vec{v}),
\end{aligned}
\end{equation*}
where $\inf$ is taken over the set $Q(\mu,\nu)$ of all probability measures $\eta$ on $\R^{d\times{d}}$, which has $\mu$ and $\nu$ as their margins.
\end{definition}

In this case, the stochastic differential equation (\ref{2.1}) has a unique solution $\mu_{t}$ which is an evolutionary measure-valued Markov process \cite{8} . Specifically, here we present an example of the coefficients in (\ref{2.1}) which satisfies Lipschitz condition with respect to space and measure-valued arguments.

\begin{example}
Suppose that the function
\begin{equation*}
\begin{aligned}
{f}:{\R^{d}}\times{\R^{d}}\rightarrow{\R^{d}},
\end{aligned}
\end{equation*}
satisfies Lipschitz condition with respect to both coordinates and is bounded. Define
\begin{equation*}
\begin{aligned}
{a}(u,\mu)=\int_{\R^{d}}{f}({u},{p})\mu(d{p}),
\end{aligned}
\end{equation*}
there exists a constant $L$, such that
\begin{equation}\label{Ex3.1.1}
\begin{aligned}
\|{a}({u},\mu)-{a}({v},\nu)\|\leq{L}(\|{u}-{v}\|+\rho(\mu,\nu)).
\end{aligned}
\end{equation}
Next, we shortly check the inequality (\ref{Ex3.1.1})
\begin{equation}\label{2.2}
\begin{aligned}
\|{a}({u},\mu)-{a}({v},\nu)\|
=&\|{a}({u},\mu)-{a}({u},\nu)+{a}({u},\nu)-{a}({v},\nu)\|\\
\leq&\|{a}({u},\nu)-{a}({v},\nu)\|+\|{a}({u},\mu)-{a}({u},\nu)\|.
\end{aligned}
\end{equation}
Firstly, there exists a constant $L_{1}$, such that
\begin{equation}\label{Ex3.1.4}
\begin{aligned}
\|{a}({u},\nu)-{a}({v},\nu)\|
=&\|\int_{\R^{d}}{f}({u},{v})\nu(d{v})-\int_{\R^{d}}{f}({v},{v})\nu(d{v})\|\\
\leq&\int_{\R^{d}}L_{1}\|{u}-{v}\|\nu(d{v})=L_{1}\|{u}-{v}\|.
\end{aligned}
\end{equation}
Secondly, taking two random variable $\xi_{1},\xi_{2}$ which have distributions $\mu$ and $\nu$,
\begin{equation*}
\begin{aligned}
&\|{a}({u},\mu)-{a}({u},\nu)\|=\|\mathbb{E}f({u},\xi_{1})-\mathbb{E}f({u},\xi_{2})\|.
\end{aligned}
\end{equation*}
Here we divide into two cases. First, assume $\|\xi_{1}-\xi_{2}\|\leq1$, by the Lipschitz condition, there exists a constant $L_{2}$, such that
\begin{equation*}
\begin{aligned}
\|f({u},\xi_{1})-f({u},\xi_{2})\| \leq{L_{2}}\|\xi_{1}-\xi_{2}\| \leq2{L_{2}}\frac{\|\xi_{1}-\xi_{2}\|}{1+\|\xi_{1}-\xi_{2}\|}.
\end{aligned}
\end{equation*}
Second, assume $\|\xi_{1}-\xi_{2}\|>1$, given the condition that $f$ is bounded, there exists a constant $L_{3}$, such that
\begin{equation*}
\begin{aligned}
&\|f({u},\xi_{1})-f({u},\xi_{2})\|
\leq L_{3}\frac{\|\xi_{1}-\xi_{2}\|}{1+\|\xi_{1}-\xi_{2}\|}.
\end{aligned}
\end{equation*}
Taking the infimum with respect to all possible joint distributions of $(\xi_{1}-\xi_{2})$, we can get there exist a constant $L_{4}$ such that
\begin{equation}\label{2.3}
\begin{aligned}
&\|{a}({u},\mu)-{a}({u},\nu)\|\leq{L_{4}}\rho(\mu,\nu).
\end{aligned}
\end{equation}
Finally, combine the inequalities (\ref{Ex3.1.4}) and (\ref{2.3}), we can get the desired inequality.
\end{example}

Next we consider equation (\ref{2.1}) with the coefficient $a$ of a special form. Suppose that
\begin{equation*}
\begin{aligned}
\R^{d}=\cup_{k=1}^{N}D_{k},
\end{aligned}
\end{equation*}
where the connected domains $D_{k}$ satisfy the condition that
\begin{equation*}
\begin{aligned}
D_{k}^{\circ}\cap D_{j}^{\circ}=\phi,\quad k\neq{j},
\end{aligned}
\end{equation*}
and the domains $D_{k}$ have piecewise smooth borders. For every $k=1,\cdots,N$ take a point $\vec{Z}_{k}\in D_{k}^{\circ}$. We define a function
\begin{equation*}
\begin{aligned}
\vec{V}\in{\mathcal{C}^{1}}(\R^{d};\R^{d}),
\end{aligned}
\end{equation*}
in such a way that
\begin{equation*}
\begin{aligned}
\vec{V}|_{\cup_{k=1}^{N}\partial{D_{k}}}=0,
\end{aligned}
\end{equation*}
 $\vec{V}$ and ${\vec{V}}'$ are both bounded on $\R^{d}$, and for every $\vec{u}\in{D_{k}^{\circ}}$ the solution to Cauchy problem
\begin{equation*}
\begin{aligned}
\left\{
    \begin{array}{lcl}
        d\vec{y}(t)=\vec{V}(\vec{y}(t))dt,\\
        \vec{y}(0)=\vec{u},
      \end{array}
    \right.
\end{aligned}
\end{equation*}
tends to $\vec{Z}_{k}$ when $t\rightarrow+\infty$. Define a non-negative function
\begin{equation*}
\begin{aligned}
h\in{\mathcal{C}^{1}}(\R^{d}),
\end{aligned}
\end{equation*}
in such a way that
\begin{equation*}
\begin{aligned}
h|_{\cup_{k=1}^{N}\partial{D_{k}}}=0,
\end{aligned}
\end{equation*}
$h$ and $h'$ are both bounded, and
\begin{equation*}
\begin{aligned}
h|_{\cup_{k=1}^{N}{D_{k}^{\circ}}}>0.
\end{aligned}
\end{equation*}
Define the function
\begin{equation*}
\begin{aligned}
\vec{f}(\vec{u},\vec{v})=h(\vec{v})\vec{V}(\vec{u}).
\end{aligned}
\end{equation*}
Then we can define the coefficient $\vec{a}$ as follows:
\begin{equation}\label{2.4}
\begin{aligned}
\vec{a}(\vec{u},\mu)=\sum_{k=1}^{N}\mathbf{1}_{D_{k}}(\vec{u})\int_{D_{k}}\vec{f}(\vec{u},\vec{v})\mu(d\vec{v}).
\end{aligned}
\end{equation}

Coefficient $\vec{a}$ is constructed to describe the motion of particles in a medium with attractors $\vec{Z}$. The next theorem shows that, starting from any mass distribution, the measure $\mu_{t}$ will be concentrated in the points $\{\vec{Z}_{k}\}^{N}_{k=1}$, as $t\rightarrow\infty$.

\begin{theorem}\label{thm2.2}
Consider equation
\begin{equation*}
\begin{aligned}
\left\{
    \begin{array}{lcl}
        d\vec{x}(\vec{u},t)=\vec{a}(\vec{x}(\vec{u},t),\mu_{t})dt,\\
        \vec{x}(\vec{u},0)=\vec{u},\quad \vec{u}\in{\R^{d}},\\
        \mu_{t}=\mu_{0}\circ{\vec{x}(\cdot,t)}^{-1},
      \end{array}
    \right.
\end{aligned}
\end{equation*}
with $\vec{a}$ of the form (\ref{2.4}). Define
\begin{equation*}
\begin{aligned}
p_{k}=\mu_{0}(D_{k}), \quad k=1,\cdots,N, \quad \nu=\sum_{k=1}^{N}p_{k}\delta_{\vec{Z}_{k}}.\\
\end{aligned}
\end{equation*}
Then
\begin{equation*}
\begin{aligned}
\rho(\mu_{t},\nu)\rightarrow0, \quad t\rightarrow+\infty.
\end{aligned}
\end{equation*}
\end{theorem}
\begin{proof}
To check the convergence in Wasserstein distance, we just need to prove that $\mu_{t}$ weakly converges to $\nu$ (see \cite{18}). For a function $\psi\in{\mathcal{C}_{b}}(\R^{d})$
\begin{equation}\label{2.5}
\begin{aligned}
\int_{\R^{d}}\psi(\vec{v})\mu_{t}(d\vec{v})=\int_{\R^{d}}\psi(\vec{x}(\vec{u},t))\mu_{0}(d\vec{u}).
\end{aligned}
\end{equation}
By the construction of the coefficient $\vec{a}$,
\begin{equation*}
\begin{aligned}
\vec{x}(\vec{u},t)\rightarrow\vec{Z}_{k},\quad {\rm when}\quad t\rightarrow+\infty, \quad \vec{u}\in{D}^{\circ}_{k},
\end{aligned}
\end{equation*}
so we have
\begin{equation*}
\begin{aligned}
\int_{\R^{d}}\psi(\vec{x}(\vec{u},t))\mu_{0}(d\vec{u})\rightarrow\int_{\R^{d}}\psi(\vec{Z}_{k})\mathbf{1}_{{D}^{\circ}_{k}}(\vec{u})\mu_{0}(d\vec{u}),\quad t\rightarrow+\infty,
\end{aligned}
\end{equation*}
which shows that $\mu_{t}$ is weakly converging to $\nu$.
\end{proof}

Next, we present a theorem about the random curve in a flow with interaction with attractors $\vec{Z}_{1},\cdots,\vec{Z}_{N}$. Specifically, we investigate the probability of a random curve having at least of type $\alpha(1),\cdots,\alpha(n)$ up to $\varepsilon$. Furthermore, we will show that this probability is less than or equal to the lower limit of itself at infinity.

\begin{theorem}
Let $\gamma_{0}(r),r\in[0,1]$ the initial random curve in $\R^{d}$, denote $\gamma_{t}=\vec{x}(\gamma_{0},t)$, where $\vec{x}$ is a solution to the equation with interaction. Let the coefficient $\vec{a}$ of the form (\ref{2.4}) with some fixed points $\vec{Z_{1}},\cdots,\vec{Z_{N}}$. Let $\varepsilon$ be sufficiently small such that
\begin{equation*}
\begin{aligned}
\forall{j}=1,\cdots,N, \quad B(\vec{Z_{k}},\varepsilon)\subset{D_{k}^{\circ}}.
\end{aligned}
\end{equation*}
We define the following probability function
\begin{equation*}
\begin{aligned}
\forall{t}\geq0 \quad R_{t}=P\{\gamma_{t} \ {\rm has} \ {\rm at} \ {\rm least}  \ {\rm type} \ \alpha(1),\cdots,\alpha(n) \ {\rm up} \ {\rm to} \ \varepsilon\}, \quad n\leq{N}.
\end{aligned}
\end{equation*}
This function $R_{t}$ has the property
\begin{equation*}
	\begin{aligned}
		\lim_{t\rightarrow+\infty}R_{t}=\sup_{s\geq 0}R_{s}.
	\end{aligned}
\end{equation*}
\end{theorem}

\begin{proof}
Define the set $A_{t}=\{\omega:\exists{r_{1}}<\cdots<r_{n}\in[0,1]:\gamma_{t}(r_{i})(\omega)\in{B}(\vec{Z}_{\alpha(i)},\varepsilon)\}$, since $\gamma_{0}$ is a random curve, which means that there exists parametrization of $\gamma_{0}$,
$$\gamma_{0}(r)(\omega), \quad r\in[0,1], \quad \gamma_{0}\in{\mathcal{C}}([0,1],\R^{d}),$$
is a random continuous function. It is known that $\vec{x}$ has a continuous modification with respect to both variables $u$ and $t$. This means that $\vec{x}$ is a random element in ${\mathcal{C}}(\R^{d}\times[0,+\infty],\R^{3})$. For fixed $r$, $\gamma_{0}(r)$ is a random variable, by the known result in \cite{20} that $\vec{x}(\gamma_{0}(r),\cdot)$ is also a random variable. Since $\gamma_{0}(\cdot)$ is continuous and $\vec{x}$ has a continuous modification, $\vec{x}(\gamma_{0}(\cdot),\cdot)$ is a random element in $\mathcal{C}([0,1])$. Fix a $\omega\in{A}_{t}$, there exist $r_{1},\cdots,r_{n}$:
$$\gamma_{s}(r_{i})(\omega)\in{B}(\vec{Z}_{\alpha(i)},\varepsilon),\quad i=1,\cdots,n,$$
By the construction of the coefficient $\vec{a}$,
$$\forall\vec{u}\in{D}^{\circ}_{k},\quad \vec{x}(\vec{u},t)\rightarrow\vec{Z}_{k},\quad t\rightarrow+\infty,\quad k=1,\cdots,N,$$
we can get
$$\gamma_{t}(r_{i})(\omega)\rightarrow\vec{Z}_{k},\quad t\rightarrow+\infty,\quad i=1,\cdots,n.$$
There exists $\tau(\omega)<+\infty$, such that
$$\tau(\omega)=\inf\{s:\forall{s'}\geq{s}\quad \forall{i}=1,\cdots,n\quad \gamma_{s'}(r_{i})(\omega)\in{B}(Z_{\alpha(i)},\varepsilon)\},$$
and $\tau(\omega)$ is a random variable. For any $t>\tau(\omega)$, $\gamma_{t}(\omega)$ has at least of type $\alpha(1),\cdots,\alpha(n)$. We choose $\tau(\omega)$ in a measurable way
$$\tau(\omega)=
    \left\{
      \begin{array}{lcl}
        0, & \omega\notin{A}_{t}, \\
        \tau(\omega), & \omega\in{A}_{t}.
      \end{array}
    \right.$$
Then for all ${s}>t$, denote
$$B_{s}=A_{t}\cap\{\omega:{s}\geq\tau(\omega)\},$$
we have $B_{s}\subset{A}_{s}$, so we have the following relation
\begin{equation*}
\begin{aligned}
R_{s}=P(A_{s})\geq{P}(B_{s}).
\end{aligned}
\end{equation*}
Since
$$\forall\omega\in{A}_{t},\tau(\omega)<+\infty,$$
by the construction of $B_{s}$, we have
\begin{equation*}
	\begin{aligned}
		B_{s}\rightarrow A_{t}, \quad s\rightarrow+\infty.
	\end{aligned}
\end{equation*}
From this, we can get
$$P(B_{s})\rightarrow{P}(A_{t}), \quad s\rightarrow+\infty.$$
Consequently,
\begin{equation*}
	\begin{aligned}
		P(A_{t})=\lim_{s\rightarrow+\infty}P(B_{s})\leq\varliminf_{s\rightarrow+\infty}P(A_{s}),
\end{aligned}
\end{equation*}
which implies
\begin{equation*}
	\begin{aligned}
		\lim_{t\rightarrow+\infty}R_{t}=\sup_{s\geq 0}R_{s}.
	\end{aligned}
\end{equation*}
\end{proof}

In Theorem \ref{thm2.2} we demonstrate that distribution of mass tends to attractor points of the differential equation with interaction. Subsequently, using this result, we will investigate the visitation measure of the random curve in a flow with interaction.

\begin{theorem}
Consider the equation with interaction with coefficient $\vec{a}$ in (\ref{2.4}). Assume the initial random curve $\{\gamma_{0}(r);r\in[0,1]\}$ has at least type $\alpha(1),\cdots,\alpha(n)$. Consider the random curve $\gamma_{t}=\vec{x}(\gamma_{0},t)$, where $\vec{x}$ is a solution to the equation with interaction. Suppose that the visitation measure $\mu^{k}_{t}$ of $\gamma_{t}$ has density $l^{k}_{t}$ with respect to the Lebesgue measure. Suppose that at initial time $t=0$,
\begin{equation*}
\begin{aligned}
\mu^{1}_{0}(\cup_{i=1}^{k}\partial{D_{i}})=0.
\end{aligned}
\end{equation*}
Denote
\begin{equation*}
\begin{aligned}
I_{1}=B(Z_{\alpha(2)}-Z_{\alpha(1)},\delta),\cdots,I_{k-1}=B(Z_{\alpha(k)}-Z_{\alpha(k-1)},\delta).
\end{aligned}
\end{equation*}
Then for every $\delta>0$ we have
\begin{equation}\label{2.23}
\begin{aligned}
\int_{\R^{d}\backslash{I_{1}}\times\cdots\times{\R}^{d}\backslash{I_{k-1}}}l_{t}^{k}(v_{1},\cdots,v_{k-1})dv_{1},\cdots,dv_{k-1}\rightarrow0,\quad a.s. \quad t\rightarrow+\infty, 
\end{aligned}
\end{equation}
\end{theorem}

\begin{proof}
First, consider the case $k=2$, consider the visitation measure for set $\{\R^{d}\backslash{I}_{1}\}$
\begin{equation*}
\begin{aligned}
\mu_{t}^{2}(\R^{d}\backslash{I}_{1})=\int_{\R^{d}\backslash{I_{1}}}l_{t}^{2}(v_{1})dv_{1}=\int_{0}^{1}\int_{0}^{1}\mathbf{1}_{\R^{d}\backslash{I_{1}}}(\vec{\xi}(t_{2})-\vec{\xi}(t_{1}))dt_{1}dt_{2},
\end{aligned}
\end{equation*}
If $(\vec{\xi}(t_{2})-\vec{\xi}(t_{1}))\in{I_{1}}$, then
\begin{equation*}
\begin{aligned}
\|(\vec{\xi}(t_{2})-Z_{\alpha(2)})-(\vec{\xi}(t_{1})-Z_{\alpha(1)})\|\leq\delta,
\end{aligned}
\end{equation*}
which exactly means the visitation measure concentrate on $B(Z_{\alpha(1)},\delta)$ and $B(Z_{\alpha(2)},\delta)$.
Consider the general case,
\begin{equation}\label{Th3.3.1}
\begin{aligned}
&\lim_{t\rightarrow+\infty}\int_{\R^{d}\backslash{I_{1}}\times\cdots\times{\R}^{d}\backslash{I_{k-1}}}l_{t}^{k}(v_{1},\cdots,v_{k-1})dv_{1},\cdots,dv_{k-1}\\
=&\lim_{t\rightarrow+\infty}\mu^{k}_{t}(\cup_{i=1}^{k-1}{\R^{d}\backslash{I}_{i}}).
\end{aligned}
\end{equation}
From Theorem \ref{thm2.2} we can get that the visitation measure $\mu^{k}$ concentrate on $B(Z_{\alpha(1)},\delta),\cdots,B(Z_{\alpha(k)},\delta)$, when $t\rightarrow+\infty$. Thus the limit of (\ref{Th3.3.1}) equals to $0$ almost surely.
\end{proof}

We described the limit behavior of a visitation measure of random curve which has at least of type $\alpha(1),\cdots,\alpha(n)$ in a flow with attractor points. To illustrate this result, we now present an example of the equation with interaction with the coefficient $\vec{a}$ in the form of (\ref{2.4}), and investigate the visitation measure of its solution.

\begin{figure}[pb]
\centerline{\includegraphics[width=8cm]{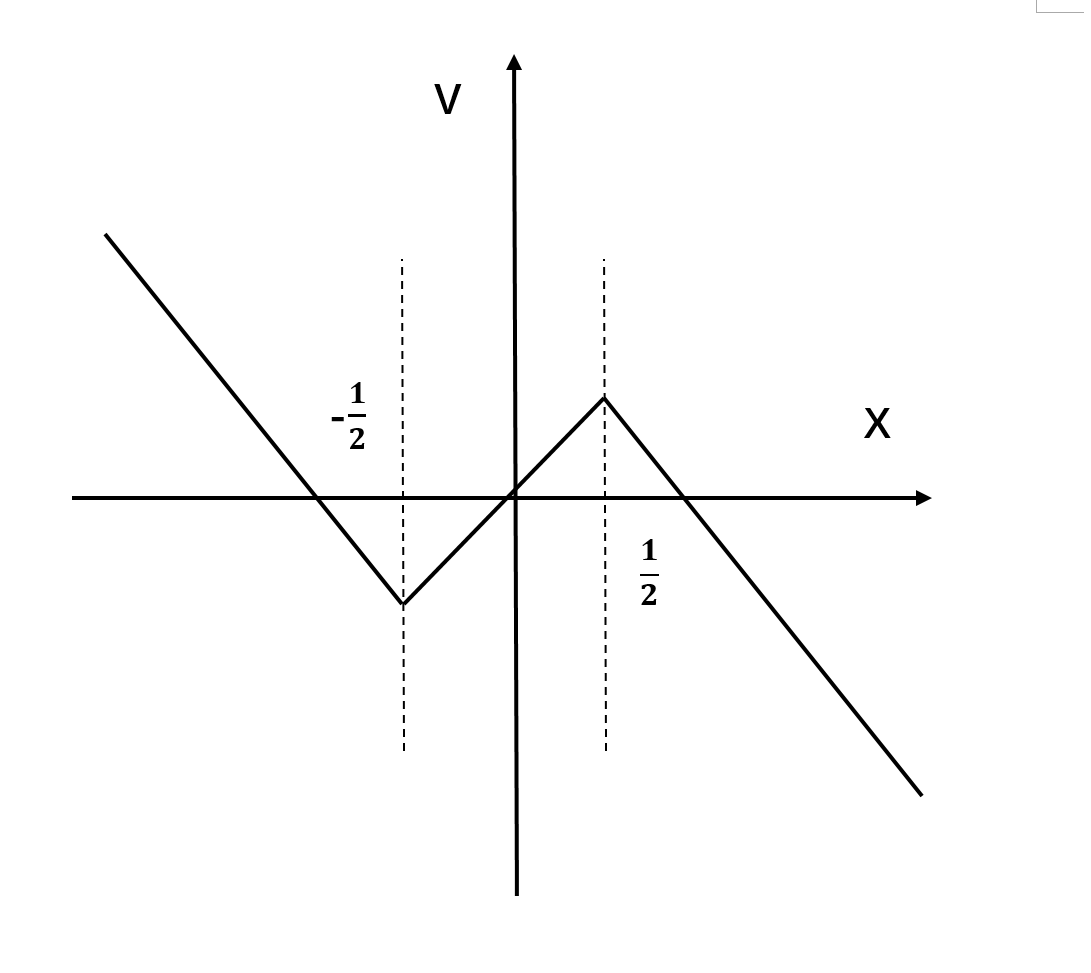}}
\vspace*{8pt}
\caption{$V$ in one dimension.}
\end{figure}

\begin{example}
When $N=4$, consider the equation with interaction with the coefficient $\vec{a}$ of the form (\ref{2.4}) in $\R^{2}$ . First, we present the graph of the function $V$ which is in one dimensional case. Then, we derive the form of $V$ which is in the two-dimensional case, and the resulting function satisfies the Lipschitz condition with respect to both coordinates. Define the function
\begin{equation*}
\begin{aligned}
{V}(x,y)=\left\{
    \begin{array}{lcl}
       (1-x,1-y), \quad (x,y)=(\frac{1}{2},\infty)\times(\frac{1}{2},\infty),\\
       (1-x,-1-y), \quad (x,y)=(\frac{1}{2},\infty)\times(-\infty,-\frac{1}{2}),\\
       (-1-x,1-y), \quad (x,y)=(-\infty,-\frac{1}{2})\times(\frac{1}{2},\infty),\\
       (-1-x,-1-y), \quad (x,y)=(-\infty,-\frac{1}{2})\times(-\infty,-\frac{1}{2}),\\
       (x,y), \quad (x,y)=[-\frac{1}{2},\frac{1}{2}]\times[-\frac{1}{2},\frac{1}{2}],\\
       (x,1-y), \quad (x,y)=[-\frac{1}{2},\frac{1}{2}]\times(\frac{1}{2},\infty),\\
       (x,-1-y), \quad (x,y)=[-\frac{1}{2},\frac{1}{2}]\times(-\infty,-\frac{1}{2}),\\
       (1-x,y), \quad (x,y)=(\frac{1}{2},\infty)\times[-\frac{1}{2},\frac{1}{2}],\\
       (-1-x,y), \quad (x,y)=(-\infty,-\frac{1}{2})\times[-\frac{1}{2},\frac{1}{2}].\\
      \end{array}
   \right.
\end{aligned}
\end{equation*}
The attractors are  ${Z}_{1}=(1,1),{Z}_{2}=(1,-1),{Z}_{3}=(-1,1),{Z}_{4}=(-1,-1)$, and $D_{1}=[0,\infty)\times[0,\infty),D_{2}=[0,\infty)\times(-\infty,0],D_{3}=(-\infty,0]\times[0,\infty),D_{4}=(-\infty,0]\times(-\infty,0]$.
For every ${u}\in{D_{k}^{\circ}}$, the solution ${x}(t)$ to Cauchy problem
\begin{equation*}
\begin{aligned}
\left\{
    \begin{array}{lcl}
        d{x}(t)={V}({x}(t))dt,\\
        {x}(0)=(u_{1},u_{2}),
      \end{array}
    \right.
\end{aligned}
\end{equation*}
tends to ${Z}_{k}$ when $t\rightarrow+\infty$.
The solution is
\begin{equation*}
\begin{aligned}
{x}(t)=\left\{
    \begin{array}{lcl}
       (u_{1}e^{t_{1}},u_{2}e^{t_{2}}), \quad (u_{1},u_{2})\in[-1,1]\times[-1,1], \quad(t_{1},t_{2})\in[0,\ln\frac{1}{|u_{1}|})\times[0,\ln\frac{1}{|u_{2}|}),\\
       (u_{1}e^{t_{1}},(u_{2}-1)e^{-t_{2}}+1), \quad (u_{1},u_{2})\in[-1,1]\times(1,+\infty), \quad t_{1}\in[0,\ln\frac{1}{|u_{1}|}),\\
       (u_{1}e^{t_{1}},(u_{2}+1)e^{-t_{2}}-1), \quad (u_{1},u_{2})\in[-1,1]\times(-\infty,-1), \quad t_{1}\in[0,\ln\frac{1}{|u_{1}|}),\\
       (u_{1}-1)e^{-t_{1}}+1,u_{2}e^{t_{2}}), \quad (u_{1},u_{2})\in(1,+\infty)\times[-1,1], \quad t_{2}\in[0,\ln\frac{1}{|u_{2}|}),\\
       (u_{1}+1)e^{-t_{1}}-1,u_{2}e^{t_{2}}), \quad (u_{1},u_{2})\in(-\infty,-1)\times[-1,1], \quad t_{2}\in[0,\ln\frac{1}{|u_{2}|}),\\
       (u_{1}-1)e^{-t_{1}}+1,(u_{2}-1)e^{-t_{2}}+1), \quad (u_{1},u_{2})\in(1,+\infty)\times(1,+\infty),\\
       (u_{1}+1)e^{-t_{1}}-1,(u_{2}-1)e^{-t_{2}}+1), \quad (u_{1},u_{2})\in(-\infty,-1)\times(1,+\infty),\\
       (u_{1}-1)e^{-t_{1}}+1,(u_{2}+1)e^{-t_{2}}-1), \quad (u_{1},u_{2})\in(1,+\infty)\times(-\infty,-1),\\
       (u_{1}+1)e^{-t_{1}}-1,(u_{2}+1)e^{-t_{2}}-1), \quad (u_{1},u_{2})\in(-\infty,-1)\times(-\infty,-1).
      \end{array}
   \right.
\end{aligned}
\end{equation*}
Define nonnegative function $h(u)$ in such a way that
\begin{equation*}
\begin{aligned}
h|_{\cup_{k=1}^{4}\partial{D_{k}}}=0,
\end{aligned}
\end{equation*}
$h$ and $h'$ are bounded, and
\begin{equation*}
\begin{aligned}
h|_{\cup_{k=1}^{4}{D_{k}^{\circ}}}>0.
\end{aligned}
\end{equation*}
Define the function
\begin{equation*}
\begin{aligned}
{f}({u},{v})=h({v}){V}({u}).
\end{aligned}
\end{equation*}
Then we can define the coefficient ${a}$ as follows:
\begin{equation*}
\begin{aligned}
{a}({u},\mu)
&=\int_{\R^{2}}h({v}){V}({u})\mu(d{v}).
\end{aligned}
\end{equation*}
Now the equation
\begin{equation*}
\begin{aligned}
\left\{
    \begin{array}{lcl}
        d{x}({u},t)={V}({x}({u},t))\int_{\R^{2}}h({v})\mu_{t}(d{v})dt,\\
        {x}({u},0)={u},\quad {u}\in{\R^{2}},\\
        \mu_{t}=\mu_{0}\circ{{x}(\cdot,t)}^{-1}.
      \end{array}
    \right.
\end{aligned}
\end{equation*}
The initial random curve $\{\gamma_{0}(r);r\in[0,1]\}$ is Brownian motion in $\R^{2} $, denote $\gamma_{t}={x}(\gamma_{0},t)$. When $k=1$, the visitation measures for $B({Z}_{1},\varepsilon)$ are
\begin{equation*}
\begin{aligned}
\mu^{1}_{0}(B({Z}_{1},\varepsilon))=\int_{0}^{1}\mathbf{1}_{B({Z}_{1},\varepsilon)}(w(r))dr,
\end{aligned}
\end{equation*}
\begin{equation}\label{2.99}
\begin{aligned}
\mu^{1}_{t}(B({Z}_{1},\varepsilon))
&=\int_{0}^{1}\mathbf{1}_{B({Z}_{1},\varepsilon)}({x}(w(r),t))dr\\
&=\lambda\{r:{x}(w(r),t)\in{B}({Z}_{1},\varepsilon)\}.
\end{aligned}
\end{equation}
For the ${V}$ and ${h}$ defined as before, if ${x}({u},0)\in{D_{1}}$, we have
\begin{equation*}
\begin{aligned}
\forall{t}\geq0, \quad {x}({u},t)\in{D_{1}}, 
\end{aligned}
\end{equation*}
then we can rewrite (\ref{2.99})
\begin{equation*}
\begin{aligned}
\lambda\{r:{x}(w(r),t)\in{B}({Z}_{1},\varepsilon)\}
=&\lambda\{r:{x}(w(r),t)\in{B}({Z}_{1},\varepsilon),w(s)\in{D_{1}}\}\\
=&\lambda\{r:w(r)\in {x}^{-1}({B}({Z}_{1},\varepsilon),t),w(s)\in{D_{1}}\}\\
=&\mu^{1}_{0}({x}^{-1}({B}({Z}_{1},\varepsilon),t)\cap{D}_{1}),
\end{aligned}
\end{equation*}
where $\vec{x}^{-1}(\Delta,t)=\{\vec{u}:\vec{x}(\vec{u},t)\in\Delta\}$.
For any ${u}\in{D}_{k}$,
\begin{equation*}
\begin{aligned}
|{x}({u},t)-{Z}_{k}|\rightarrow0, \quad t\rightarrow\infty.
\end{aligned}
\end{equation*}
Consequently,
\begin{equation*}
\begin{aligned}
&\lim_{\varepsilon\rightarrow0}|\mu^{1}_{t}(B({Z}_{1},\varepsilon))-\mu^{1}_{0}(D_{1})|\\
=&\lim_{\varepsilon\rightarrow0}|\mu^{1}_{0}({x}^{-1}({B}({Z}_{1},\varepsilon),t)\cap{D}_{1})-\mu^{1}_{0}(D_{1})|=0.
\end{aligned}
\end{equation*}

\end{example}

\end{document}